\documentclass{cmslatex}
\usepackage{latexsym, amssymb, enumerate, amsmath}
\usepackage{pgfplots}

\newcommand{\Sym}{\text{Sym}}
\newcommand{\TV}{\text{TV}} 
\newcommand{\equnder}[1]{\underset{\eqref{#1}}{=}}
\newcommand{\lequnder}[1]{\underset{\eqref{#1}}{\leq}}
\newcommand{\Cspace}{{\mathcal{C}}_c^\infty}
\newcommand{\dualball}[1]{{\mathcal{B}}^*_{{#1}}}
\newcommand{\lr}[1]{\left( #1 \right)} 
\newcommand{\lre}[1]{\left[ #1 \right]} 

\newcommand{\Rp}{\R \cup \{+\infty\}} 
\newcommand{\TGV}{\text{TGV}_{\lambda_1,\lambda_2}}
\newcommand{\F}{{\mathcal F}}
\newcommand{\G}{{\mathcal G}}
\newcommand{\Sim}{{\mathcal S}}
\newcommand{\Reg}{{\mathcal R}}
\newcommand{\ud}{u^\delta}
\newcommand{\abs}[1]{\left| #1 \right|} 
\newcommand{\norm}[1]{\left\| #1 \right\|} 
\newcommand{\dualnorm}[2]{\left\| #1 \right\|_{*,#2}}
\newcommand\inner[2]{\left\langle #1,#2 \right\rangle} 
\newcommand\set[1]{\left\{ #1 \right\}} 
\newcommand{\N}{\mathbb{N}} 
\newcommand{\R}{\mathbb{R}} 

\newcommand{\utgv}{u_{\lambda_1,\lambda_2}} 
\newcommand{\tvm}[1]{{v_{\lambda_{#1}}^{#1}}}
\newcommand{\tvmm}[1]{{v_{\mu_{#1}}^{#1}}}
\newcommand{\tvmast}[1]{\tvm{#1}\!\!\!{}^*}

\newcommand{\um}{u_{\min}}

\newtheorem{thm}{Theorem}[section]

\newtheorem{example}[thm]{Example}

\newtheorem{remark}[thm]{Remark}

\title{Exact Solutions of One-Dimensional TGV
}

\author{Christiane P\"oschl
\thanks{Department of Mathematics, 
Alpen Adria Universi\"at, Klagenfurt, Austria, 
(Christiane.Poeschl@aau.at).
}
\and Otmar Scherzer 
\thanks{Computational Science Center, 
University of Vienna,
Nordbergstra\ss{}e~15, A-1090 Wien, 
Austria and Johann Radon Institute of
Computational and Applied Mathematics,
Austrian Academy of Sciences, Altenberger Stra\ss{}e~69, A-4040 Linz, Austria,
(Otmar.Scherzer@univie.ac.at).}
}

\begin{document} 

\maketitle 

\begin{abstract} 
$L^2$-TGV-regularization has been introduced by Bredies, Kunisch, and Pock
\cite{BreKunPoc10}. This regularization method requires careful tuning of 
\emph{two} regularization parameters. The focus of this paper is to derive analytical 
results, which allow for characterizing parameter settings, which make this 
method in fact different from $L^2-\TV$ (the ROF-model) and $L^2-\TV^2$ 
regularization, respectively.
In this paper we also provide explicit solutions of TGV-denoising for 
particular
one-dimensional function data.
\end{abstract}

\section{Introduction}
\label{sec:intro}
TGV-denoising has been introduced in \cite{BreKunPoc10}: Given $k \in \N_0$ 
and a  function $\ud: \Omega \to \R$, 
where $\Omega \subset \R^d$, the method consists in determining 
\begin{equation}
\label{eq:L2TGV}
u_{\vec{\lambda}}:= \text{argmin} \set{\mathcal{G}_{\vec{\lambda}}^k(u): u \in L^2(\Omega)}\,,
\end{equation}
where
\begin{equation}
\label{eq:tgv_min}
\mathcal{G}_{\vec{\lambda}}^k(u):=
\frac{1}{2}\int_\Omega (u-\ud)^2\,dx\; + 
\; TGV_{\vec{\lambda}}^k(u)
\end{equation}
with
\begin{equation*}
\begin{aligned}
 ~ & \text{TGV}_{\vec{\lambda}}^k(u)\\
 = & \sup 
 \set{\int_\Omega u (\nabla \cdot)^k \phi\,dx : \phi \in 
\Cspace(\Omega,\Sym^k(\R^d)), \norm{(\nabla \cdot)^{k-l} \phi}_{L^\infty}\leq
\lambda_l,\; l=1,\dots,k}\,,
\end{aligned}
\end{equation*}
and $\Sym^k(\R^d)$ denotes the space of symmetric tensors of order $k$ with
arguments in $\R^d$. 
There can be imagined several realizations of $\norm{(\nabla 
\cdot)^{k-l}v}_{L^\infty}$ to be implemented - one of them is 
$\norm{(\nabla \cdot)^{k-l}\phi}_{L^\infty} = \sup\set{\abs{(\nabla 
\cdot)^{k-l}\phi(x)}_{l^2}: x \in \Omega}$, where $\abs{\cdot}_{l^2}$ 
denotes the Frobenius-norm of a tensor. Note that the definition here is 
slightly different to \cite{BreKunPoc10}, where in the original 
definition, the enumeration of the indices of $\lambda_i$ is reversed. 

All along this paper, for the simplification of notation and considerations, we 
restrict attention to the case $k=2$. Consequently, from now on, we omit the superscript $k$ 
in the TGV-functional.

The goal of this paper is to increase the knowledge about structural properties 
of TGV-denoising, and to put this method into perspective with 
total variation and second order total variation regularization by analytical 
means. This is done in two different ways:
\begin{enumerate}
\item The main result of this paper concerns the characterization of the sets of 
regularization parameters 
      $\vec{\lambda} = (\lambda_1,\lambda_2)$ such that the minimizers of $\TGV$ 
either equal total variation minimizers 
      or minimizers of the second order total variation minimization, and to 
determine sets of parameters, where TGV minimization is in 
      fact different from either one of them.
\item We study analytical solutions of simple one-dimensional test-cases, where
      $d=1$, $\Omega = (-1,1)$, and $k=2$. In this simple situation TGV-denoising 
\eqref{eq:L2TGV} simplifies to minimizing the functional
      \begin{equation}\label{eq:l2tgv}
      \begin{aligned}
      \G: L^2(-1,1) & \rightarrow \Rp\,,\\
        u & \to \G(u):=\frac{1}{2}\int_{-1}^1 (u-\ud)^2\,dx\; + \; \TGV(u)\;.
      \end{aligned}
      \end{equation}
      In the specific one-dimensional situation $\TGV$ can be written as
 \begin{equation} \label{eq:tgv_1d}
 \begin{aligned}
 ~ & \TGV(u) \\
 & = \sup \set{\int_{-1}^1 u \phi''\,dx: \phi \in  \Cspace(-1,1)\,,
 \norm{\phi'}_{L^\infty} \leq \lambda_1\,, \norm{\phi}_{L^\infty} \leq 
\lambda_2}\;.
 \end{aligned}
 \end{equation}
Similar, as in our previous work \cite{PoeSch08} for total variation
minimization and minimization with totally bounded second derivative, 
it is possible to characterize the minimizers of $\TGV$ in a simple manner using
Fenchel-duality theory. We show that the minimizers 
are either equal to $\ud$ or piecewise affine linear 
that bend or jump, whenever
the first or second primitives of the dual functions 
attains an extremum.

We then study explicit solutions of TGV-denoising for the basic test data 
cases 
\begin{equation*}
x \to \ud(x) =\abs{x}-\frac{1}{2}\,,
\end{equation*}
\begin{equation*}
x \to \ud(x) = {\bf 1}_{[-1/2,1/2]}(x)-\frac{1}{2}\,,
\end{equation*}
and 
\begin{equation*}
x \to \ud(x) = x^2-\frac{1}{3}\,.
\end{equation*}
For the first two exemplary cases the minimizers of the TGV-functional
\eqref{eq:l2tgv} are weighted sums of $\TV$-minimizers and $\TV^2$-minimizers.
The second example has also been studied in \cite{BenBruBurMue13} - however, 
there no complete characterization of the parameter sets have been stated where the 
$\TGV$-minimizer equals either $L^2-\TV$, $L^2-\TV^2$-minimizers, which is a focus topic of this work.
\end{enumerate}

The outline of this paper is as follows:
In Section \ref{sec:notation} we introduce preliminary notation and the main
definitions. We derive characteristic properties of minimizers of the
TGV-denoising problem (in $d$-dimensions) via convex duality theory (Sections \ref{sec:fenchel}, \ref{sec:homogeneous}). Later we restrict attention to the case $d=1$ 
and show that minimizers are either equal to the data or piecewise affine linear (cf. Section \ref{sec:characterization}).
Finally we calculate explicit minimizers for the $\TGV$-functional in the
case where the data are the absolute value (Section \ref{se:ex1}), 
the indicator function (Section \ref{se:ex2}), or a quadratic polynomial 
(Section \ref{se:ex3}), respectively.

\section{Notation} 
\label{sec:notation}
Let $\Omega \subseteq \R^d$ be a bounded, connected domain with Lipschitzian boundary. Moreover, let $\ud : \Omega \to \R \in L^2(\Omega)$. 
For $i \in \N$ we define the following functional:
\begin{equation}
\label{eq:l2TV}
\begin{aligned}
\F^i: L^2(\Omega) &\rightarrow \Rp\,,\\
u &\to \F^i(u):=\frac{1}{2}\int_\Omega (u-\ud)^2\,dx\; + \TV_{\lambda_i}^i(u)
\end{aligned}
\end{equation}
where 
\begin{equation}
\label{eq:def_lambda}
\begin{aligned}
 ~ & \TV_{\lambda_i}^i(u)\\
 := & \sup \set{ \int_\Omega u \phi^{(i)}\,dx : \phi \in 
\Cspace(\Omega,\Sym^k(\R^d)), \norm{(\nabla \cdot)^i \phi}_{L^\infty} \leq \lambda_i}\,,
\end{aligned}
\end{equation}
where 
\begin{equation*}
\norm{(\nabla \cdot)^{k-l}v}_{L^\infty} = \sup\set{\abs{(\nabla 
\cdot)^{k-l}v(x)}_{l^2}: x \in \Omega}\,, 
\end{equation*}
and $\abs{(\nabla \cdot)^i \phi}_{l^2}$ denotes the Frobenius-norm of $(\nabla \cdot)^i \phi$.

The minimizer of \eqref{eq:l2TV} is denoted by $\tvm{i}$.
The minimizer of \eqref{eq:tgv_min} is denoted by $\utgv$.

Because 
\begin{equation}
 \label{eq:TV_lambda}
 \TV_{\lambda_i}^i (u) = \lambda_i \TV^i(u)\,,
\end{equation}
we see that minimization of the functional $\F^i$ from \eqref{eq:l2TV} is standard $L^2-\TV$-minimization with regularization parameter $\lambda_i$. 
$L^2-\TV$-minimization has been studied widely in the literature. In the 
one-dimensional $d=1$ setting it is used for regression (see e.g. \cite{MamGee97,DavKov01}) - analytical solutions have been calculated for instance in \cite{ChaEse05}. 
In image processing, for $d \geq 2$,  $L^2-\TV$-regularization it is called the Rudin-Osher-Fatemi 
model \cite{RudOshFat92}. Regularization with derivatives of higher order 
bounded variation has been studied for instance in 
\cite{Ste06,Poe08,PoeSch08,SchGraGroHalLen09}.

\section{Fenchel duality and applications}
\label{sec:fenchel}

In the following let $\mathcal{H}$ be a Hilbert-space. 
In this case it is common to identify $\mathcal{H}$ with its dualspace 
and to identify the dual pairing $\inner{u^*}{u}$ on $\mathcal{H}^*$ and 
$\mathcal{H}$ with the inner product on $\mathcal{H}$. 
For instance when $\mathcal{H}=L^2(\Omega)$, $\inner{u^*}{u} = \int_{\Omega} u^* u\; dx$.

We start by defining the $\ast$-number, which is a generalization of the 
dualnorm of a Banach-space,
to convex, positively homogeneous functionals.

\begin{definition}
A proper, convex functional ${\mathcal T}: \mathcal{H} \to \Rp$ is positively 
homogenuous, if there exists some 
$l=1,2\ldots$ such that ${\mathcal T}: \mathcal{H} \to \Rp$ is $l$-homogeneous, 
which means that
 \begin{equation*}
  {\mathcal T}(\lambda u) = \abs{\lambda}^l {\mathcal T}(u)\,, \quad \forall 
\lambda \in \R\;.
 \end{equation*}
\end{definition}

\begin{definition}[The $\ast$-number]\label{de:astnumber}
Let $\Reg: \mathcal{H} \rightarrow \Rp$ be a positively homogeneous and convex 
functional .
For $u^* \in \mathcal{H}$ define
\begin{equation*}
 \dualnorm{u^*}{\Reg}:=
\sup \set{\inner{u^*}{u}: u\in \mathcal{H}, \Reg(u)\leq 1}\,.
\end{equation*}

Moreover define 
\begin{equation*}
\dualball{\Reg}:=\set{u^* \in \mathcal{H}: \dualnorm{u^*}{\Reg}\leq 1}
\end{equation*}
as the dualball with respect to the $\ast$-number.
\end{definition}

\begin{example}
From \eqref{eq:TV_lambda} it follows that 
\begin{equation}
 \label{eq:ast_tv}
 \dualnorm{u^*}{\TV_{\lambda_i}^i} = \frac{1}{\lambda_i} \dualnorm{u^*}{\TV_{1}^i}\;.
\end{equation}
Note that according to our definition 
\begin{equation*}
 \TV_1^i = \TV^i\,,\qquad \forall i \in \N\,.
\end{equation*}
\end{example}

\begin{lemma}

\label{le:perp}
(see \cite[Lemma 4.6]{SchGraGroHalLen09}).
 Let $\Reg$ be positively homogeneous and set
\begin{equation*}
 \mathcal{P}:=\set{p \in  \mathcal{H}: \Reg(p) =0}\;.
\end{equation*}
From the assumptions that $\Reg$ is positively homogeneous and convex, it follows that 
$\mathcal{P}$ is a linear subspace of $\mathcal{H}$. Denote by
\begin{equation*}
 \mathcal{P}^\perp:=\set{u^*\in \mathcal{H}: 
\inner{u^*}{p}=0,p\in \mathcal{P}}\;.
\end{equation*}
Then $\dualnorm{u^*}{\Reg}=+\infty$ for all $u^* \not \in \mathcal{P}^\perp$.

\end{lemma}

\begin{definition}
Assume that $i=1,2,\ldots$.
Let $\mathcal{H}=L^2(\Omega)$, $\Reg=\TV_{\lambda_i}^i$, and let $\mathcal{P}^i$ be the 
set of polynomials of order $i-1$. Then 
\begin{equation}
\label{Hspace*}
 \mathcal{H}^i := \set{u \in L^2(\Omega) : \int_\Omega u(x) x^j \,dx = 0,\, 
\abs{j}=0,1,\ldots,i-1,\, j \in \N_0^d} = \mathcal{P}^i{}^\bot\;.
\end{equation}
\end{definition}
\begin{remark}
Because $C_c^\infty(\Omega)$ is dense in $L^2(\Omega)$, and $\Omega$ is assumed to be bounded, it follows that
 \label{re:hspace*}
 \begin{equation*}
  \mathcal{H}^i := \overline{\set{u \in C_c^\infty(\Omega) : \int_\Omega u(x) x^j \,dx = 0,\, 
\abs{j}=0,1,\ldots,i-1,\, j \in \N_0^d}}\;.
\end{equation*}
\end{remark}

The following lemma is a direct consequence of 
Lemma \ref{le:perp} and the above definition: 
\begin{lemma} \label{re:invarianz}
\begin{itemize}
 \item{$\TV_{\lambda_i}^i$, $i=1,2$:} For all $u^* \not \in \mathcal{H}^i$, $\dualnorm{u^*}{\TV_{\lambda_i}^i}=+\infty$. 
\item{TGV-functional:} For all $u^* \not \in \mathcal{H}^2$,
$\dualnorm{u^*}{\TGV}= +\infty$. 
\end{itemize}
\end{lemma}

The definitions of the $\dualnorm{\cdot}{\TV_{\lambda_i}^i}$-norms are similar as in Meyer's 
book 
\cite{Mey01}. The difference is that there $\Omega = \R^d$ is considered, and 
the elements of the space $L^2(\R^d)$ satisfy natural boundary conditions at 
$\infty$. 
Since we consider bounded domains $\Omega$ we restrict attention to the 
subspaces 
$\mathcal{H}^i$ rather than to $L^2$, as in Meyer's book. 
Another possibility, instead of factorizing out polynomials, is to consider 
boundary conditions on the bounded domain $\Omega$, which has been realized in 
\cite{AubAuj05}.

The Fenchel dual of a proper functional $\mathcal{S}: \mathcal{H} \rightarrow 
\Rp$ 
is defined as
\begin{equation*}
\begin{aligned}
 \mathcal{S}^*: \mathcal{H} &\rightarrow \Rp\;.\\
u^* & \mapsto \Sim^*:=\sup_{u\in \mathcal{H}}\set{ \inner{u^*}{u} - \Sim(u)}\;.
\end{aligned}
\end{equation*}

The following results can be found in \cite{EkeTem76}, see also \cite{SchGraGroHalLen09}:
\begin{remark}
\label{re:2.5}
\begin{itemize}
\item Let ${\mathcal T}$ be $1$-homogeneous, then the Fenchel dual function is a 
characteristic function of a convex set $\mathcal{C}^*$. That is, 
 \begin{equation}
 \label{eq:char}
  {\mathcal T}^*(u^*) = \chi_{{\mathcal C}^*}(u^*) = 
  \left\{ \begin{array}{rl} 0 & \text{ for } u^* \in {\mathcal C}^*\,,\\
                                                     +\infty & \text{ else }\;.
                                                    \end{array} \right.
\end{equation}
In particular, for ${\mathcal T}$ $1$-homogeneous,
\begin{equation}
\label{eq:relation_dualball}
 \mathcal{C}^*= \dualball{\mathcal{T}}\;.
\end{equation}
\item Let $\Sim, \Reg$ be convex and proper functionals defined on $\mathcal{H}$.
Denote by $\hat{u}$ a minimizer of the functional $u \to \Sim(u) + \Reg(u)$ and denote by 
$\hat{u}^*$ a minimizer of 
the functional $u^* \to \Sim^*(u^*) + \Reg^*(-u^*)$. 
Then the extremality  conditions holds:
\begin{equation}\label{eq:extremality2}
\hat {u}^*\in \partial \Sim(u) 
\qquad \text{ and }\qquad
-\hat {u}^*\in \partial \Reg(u).
\end{equation}
\end{itemize}
\end{remark}

\begin{example}
\label{ex:dual}
\begin{enumerate}
\item The dual functional of 
\begin{equation}
\label{eq:s}
\begin{aligned} 
 \Sim: L^2(\Omega) &\to \R\,,\\
 u & \to \frac{1}{2} \int_\Omega(u-\ud)^2\,dx\,,
\end{aligned}
\end{equation}
is given by 
\begin{equation}
\label{eq:s*}
\begin{aligned} 
 \Sim^*: L^2(\Omega) &\to \R\,,\\
  u^* &\to \frac{1}{2} \int_\Omega u^*{}^2\,dx + \int_\Omega u^* \ud \,dx\;.
\end{aligned}
\end{equation}
In the case of the quadratic functional the extremality condition
\eqref{eq:extremality2} for a
minimizer shows:
\begin{equation}
\label{eq:stern_and_u}
 \hat{u}^* = \hat{u}-\ud.
\end{equation}
\item Let $$\Reg = \TV_{\lambda_i}^i : L^2(\Omega) \to \Rp\,,$$ which is $1$-homogeneous. 
Then 
 \begin{equation*}
  \Reg^*(u^*) = \sup_{u \in L^2(\Omega)} \set{ \int_\Omega u^* u\,dx - 
\TV_{\lambda_i}^i(u)} 
                     \underset{\eqref{eq:TV_lambda}}{=} \chi_{\dualball{\lambda_i \TV^i}}(u^*)\,,
 \end{equation*}
 where the characteristic function is $0$ on the closed unit ball
 ${\mathcal C}^* = \dualball{\lambda_i \TV^i}$ and $+\infty$ else. 
\item Let $$\Reg = \TGV: L^2(\Omega) \to \Rp\,,$$ which is $1$-homogeneous 
too. Thus $\Reg^*(u^*) = \chi_{\dualball{\TGV}}$.
\end{enumerate}
\end{example}

\section{Regularization Methods with $1$-Homogeneous Regularizers}
\label{sec:homogeneous}
In the following we derive some properties of regularization functionals with $\Sim$ from 
\eqref{eq:s} and $1$-homogeneous regularizers $\Reg$. We denote by 
\begin{equation*}
 u \to {\mathcal E}(u): = \Sim(u) + \Reg(u)\;.
\end{equation*}

\begin{remark}\label{re:dual}
From \eqref{eq:relation_dualball} we know that 
$$\Reg^* = \chi_{\dualball{\Reg}}\;.$$
Then the extremality condition \eqref{eq:stern_and_u} guarantees that 
$\hat{u}^* = \hat u-\ud \in \dualball{\Reg}$ 
and from Fenchel-duality theory we see that
\begin{equation}
\label{eq:help}
\begin{aligned}
& \frac{1}{2} \int_\Omega (\hat{u}-\ud)^2\,dx + \Reg(\hat{u})\\
= & \inf \set{\Sim(u) + \Reg(u)}\\
= & - \inf \set{\Sim^*(u^*) + \Reg^* (-u^*)} \\
= & - \Sim^*(\hat{u}^*) - \underbrace{\chi_{\dualball{\Reg}}(-\hat{u}^*)}_{=0}\\
\underbrace{=}_{\eqref{eq:s*}} & - \int_\Omega \frac{1}{2} \lr{\hat{u}-\ud}^2 + \ud (\hat{u} -\ud)\,dx\;.
\end{aligned}
\end{equation}

In summary we have shown that 
\begin{equation}
 \label{eq:opti}
 \begin{aligned}
 \dualnorm{\hat u^*}{\Reg} &\leq 1\,,\\
\Reg(\hat{u}) &= 
- \int_\Omega  
\lr{\hat{u}-\ud}^2 + \ud \lr{\hat{u} -\ud}
\,dx\; =-\int_\Omega 
\lr{\hat{u}-\ud} \hat{u}\,dx\;.
\end{aligned}
\end{equation}
\begin{itemize}
\item 
\eqref{eq:opti} applied to $L^2-\TV$-minimization, shows that 
$\tvm{i}$, the minimizer of $\F^i$ (see \eqref{eq:l2TV}), 
satisfies  
\begin{equation}\label{eq:optiTV}
\dualnorm{\tvmast{i}}{\TV_{\lambda_i}^i} \leq 1 \text{ and }
\TV_{\lambda_i}^i(\tvm{i}) = - \int_\Omega (\tvm{i}-\ud) \tvm{i}\,dx\;.
\end{equation}
$\dualnorm{\tvmast{i}}{\TV_{\lambda_i}^i}\leq 1$ implies that $\tvmast{i} \in 
\mathcal{H}^i$, because 
according to Lemma \ref{le:perp} $\dualnorm{\tvmast{i}}{\TV_{\lambda_i}^i}$ would be $+\infty$ 
otherwise.
This, in particular, means that if $\ud \in \mathcal{H}^i$, then also 
$\tvm{i} \in \mathcal{H}^i$.
\item
\eqref{eq:opti} applied to $\TGV$-minimization, shows that $\utgv$, the 
minimizer of \eqref{eq:l2tgv}, satisfies  
\begin{equation}\label{eq:optiTGV}
\dualnorm{\utgv^*}{\TGV} \leq 1 \text{ and }
\TGV(\utgv) = - \int_\Omega (\utgv-\ud) \utgv\,dx\;.
\end{equation}
Now 
$\dualnorm{\utgv^*}{\TGV}\leq 1$
implies that $\utgv^* \in \mathcal{H}^2$ (Lemma \ref{le:perp}).
Since $\utgv^* = \utgv-\ud$, $\ud \in \mathcal{H}^2$ implies that also
$\utgv \in \mathcal{H}^2$, hence by assuming that 
$\ud \in \mathcal{H}^2$ we can equivalently minimize $\G$ over $\mathcal{H}^2$ 
instead of $L^2(\Omega)$. 
\end{itemize}
\end{remark}

\begin{lemma}\label{le:min}
If $u$ satisfies \eqref{eq:opti} then $u$ minimizes ${\mathcal E}$.

In particular 
\begin{itemize}
\item $u$ minimizes $\F^i$ iff $- \int_\Omega (u-\ud) u \,dx = \TV_{\lambda_i}^i(u)$ and 
$u-\ud$ in $\dualball{\TV_{\lambda_i}^i}$, 
\item and $u$ minimizes the TGV-functional iff $- \int_\Omega (u-\ud) u \,dx = \TGV(u)$ and 
$u-\ud$ in $\dualball{\TGV}$. 
\end{itemize}
\end{lemma}

\begin{proof}
We prove the lemma by contradiction: Assume that $u$ satisfies the assumptions 
of the lemma but is not a minimizer of ${\mathcal E}$. Then there exists some $v 
\not= u$ 
such that $v$ minimizes ${\mathcal E}$ and ${\mathcal E}(v)<{\mathcal E}(u)$. 
From 
\eqref{eq:optiTGV} it then follows that
\begin{equation*}
\Reg(v) = -\int_\Omega (v-\ud)v\,dx\;.
\end{equation*}
Therefore, from the assumption that $u$ satisfies \eqref{eq:opti}, we see that
\begin{equation*}
\begin{aligned}
~ & \frac{1}{2} 
\int_\Omega \lr{v-\ud}^2 \, dx - \int_\Omega (v-\ud) v\,dx \\
= & \quad 
\frac{1}{2} \int_\Omega \lr{v-\ud}^2\,dx + \Reg (v)\\
< & \quad \frac{1}{2} \int_\Omega \lr{u-\ud}^2\,dx + \Reg(u)\\
= & \quad
\frac{1}{2} \int_\Omega \lr{u-\ud}^2 \, dx - \int_\Omega (u-\ud) u\,dx
\end{aligned}
\end{equation*}
such that 
\begin{equation}\label{eq:contra1}
-\frac{1}{2} \int_\Omega v^2 dx<
-\frac{1}{2} \int_\Omega u^2 dx\;.
\end{equation}
The dual functional of a convex, $1$-homogeneous function $\Reg$, is the 
characteristic function of $\dualball{\Reg}$ (cf. Remark \ref{re:dual}).
The Fenchel-duality theorem (see e.g. \cite{EkeTem76}) states, that 
$v^* := v-\ud$ minimizes the functional 
$w^*\rightarrow \Sim^*(w^*)$ over $\dualball{\Reg}$, 
where $\Sim^*$ is as in 
\eqref{eq:s*}, such that we have now
\begin{align*}
      \Sim^*(v-\ud) 
&=    \int_\Omega \lr{\frac{1}{2}(v-\ud)^2 + (v-\ud)\ud}\,dx\\
&\leq \Sim^*(u-\ud) 
=     \int_\Omega \lr{\frac{1}{2}(u-\ud)^2 + (u-\ud)\ud}\,dx\;.
\end{align*}
The inequality above simplifies to
\begin{equation*}
\frac{1}{2} \int_\Omega v^2 dx\leq
\frac{1}{2} \int_\Omega u^2 dx\;, 
\end{equation*}
such that we obtain a contradiction to \eqref{eq:contra1}.
Hence the assumption that $v \not =u$ is a minimizer of ${\mathcal E}$ was 
wrong.
\end{proof}

\begin{lemma}\label{le:null}
Assume that $\Reg$ is 1-homogeneous functional on  $\mathcal{H}$. Then 
$\um\equiv 0$ minimizes $\mathcal{E}$ if and only if $\dualnorm{\ud}{\Reg} \leq 1$.
\end{lemma}

\begin{proof}
\begin{itemize}
\item $0$ minimizes $\mathcal{E} \Rightarrow \dualnorm{\ud}{\Reg}\leq 1$: 
If $\um\equiv 0$, then $\um^*=-\ud$ and the extremality conditions from Remark \ref{re:dual} 
state that 
$\um^* \in \dualball{\Reg}$.
This means that $\dualnorm{\um^*}{\Reg}\leq 1$
and consequently $\dualnorm{\ud}{\Reg}\leq 1\;$.

\item  $ \dualnorm{\ud}{\Reg}\leq 
1\Rightarrow \um\equiv 0 $ :
We prove this implication by contradiction.
Assume therefore that
$\dualnorm{\ud}{\Reg} \leq 1$
and that $\um\not\equiv 0$ minimizes $\mathcal{E}$. 
This, in particular, means that $\Reg(\um) < +\infty$.
Then from \eqref{eq:opti} it follows that
\begin{align*}
  -\int_\Omega  \um (\um-\ud) \; dx &= \Reg(\um)
   \geq \dualnorm{\ud}{\Reg}  \Reg(\um)\;.
\end{align*}
Rearranging the terms and division by $\Reg(\um)$ shows that
\begin{equation}
\label{eq:1b}
\begin{aligned}
  & - \int_\Omega  \frac{\um^2}{\Reg(\um)}\, dx +
     \int_\Omega  \frac{\um}{\Reg(\um)}\ud  \, dx \\
\geq & \dualnorm{\ud}{\Reg} \\
 = & \sup\set{\int_\Omega   \ud \phi \, dx: \Reg(\phi) \leq 1 }\;.
\end{aligned}
\end{equation}
Since, by assumption, $0 \neq \um \in L^2(\Omega)$, we also have 
\begin{equation*}
\int_\Omega  \frac{\um^2}{\Reg(\um)}\, dx >0\;.
\end{equation*}
This, together with \eqref{eq:1b}, shows that
\begin{align*}
\sup\set{\int_\Omega   \ud \phi \, dx: \Reg(\phi) \leq 1 } 
  &\underbrace{\geq}_{\phi = \frac{\um}{\Reg(\um)}} \int_\Omega  \frac{\um}{\Reg(\um)}\ud  \, dx \\
  &>\dualnorm{\ud}{\Reg}\\ 
  &=  \sup\set{\int_\Omega   \ud \phi \, dx: \Reg(\phi) \leq 1 }\;,
\end{align*}
hence we obtain a contradiction to the assumption $\um \neq 0$.
\end{itemize}
\end{proof}

\begin{example}
\label{le:gnorm}
\begin{enumerate}
 \item Let $\Reg = \TV_{\lambda_i}^i$, then from Lemma \ref{le:null} it follows that $\tvm{i} \equiv 0$ if and only if  $\dualnorm{\ud}{\TV_{\lambda_i}^i} \leq 1$. 
 \item $\TGV$-minimization: Choose $\Reg = \TGV$, then from Lemma \ref{le:null} it follows that $\utgv \equiv 0$ if and only if 
 $\dualnorm{\ud}{\TGV} \leq 1$\;.
\end{enumerate}
These results are similar as in \cite{Mey01}, where $TV$-minimization of functions on $\Omega=\R^d$ have been considered.
\end{example}

\section{Extremal Properties and Solutions of 1D-TGV}
\label{sec:characterization}
In the following we consider the case $d=1$ and $\Omega = (-1,1)$.
We derive some characteristic properties of the minimizers 
$\utgv$ of the $\TGV$-functional $\G_{\lambda_1,\lambda_2}$, defined in \eqref{eq:tgv_min}.

Below, by some basic considerations, it is possible to identify sets of 
parameters $\vec{\lambda} = (\lambda_1,\lambda_2)$ for which $\utgv$ equals some $\tvm{i}$, $i=1,2$.

For $d=1$, the dual-norm $\dualnorm{\cdot}{\TV_{\lambda_i}^i}, i=1,2$, $\dualnorm{u^*}{\TGV}$, respectively, 
can be easily calculated via integration: To see this, let 
\begin{equation}
\begin{aligned}
 \label{eq:sigma}
 \sigma^0[u^*](x) &:= u^*(x)\,,\\
\sigma^1[u^*](x) &:=\int_{-1}^x u^*(t)\,dt\,,\\
\sigma^{i+1}[u^*](x) &:=\int_{-1}^x
\sigma^i[u^*](t)\,dt\,,  \qquad \forall u^* \in L^2(-1,1)\;.
\end{aligned}
\end{equation}

\begin{lemma}
 \label{le:hspace*}
 Let $\Omega = (-1,1)$. Then for all $i=1,2,\ldots$,
 \begin{equation}
 \label{eq:Psi}
 \begin{aligned}
 \Psi^i &:= \set{\psi \in C_c^\infty(-1,1): \sigma^j[\psi](1)=0\,,j=1,2,\ldots,i-1}\\
   &= \mathcal{H}^i \cap C_c^\infty(-1,1)\;. 
 \end{aligned}
 \end{equation}
 Moreover, 
  \begin{equation}
  \label{eq:density}
  \overline{\Psi^i} = \mathcal{H}^i\;. 
 \end{equation}
\end{lemma}
\begin{proof}
Let $u^* \in \mathcal{H}^i \cap C_c^\infty(-1,1)$, then 
\begin{equation*}
 \int_{-1}^1 u^* x^j \,dx = 0\,, \qquad \forall j=0,1,\ldots i-1\;.
\end{equation*}
For fixed $i$ we prove by an inductive argument that for $u^* \in \mathcal{H}^i \cap C_c^\infty(-1,1)$ also $u^* \in \Psi^i$.
\begin{itemize}
 \item Let $j=1$: Then $u^* \in \mathcal{H}^i$ implies that $$\sigma^j[u^*](1) = \sigma^1[u^*](1) = \int_{-1}^1 u^* dx = \int_{-1}^1 u^* 1dx =  0\;.$$
 \item Let $2 \leq j \leq i-1$ and assume that $\sigma^k[u^*](1) =0$ for $k=0,1,\ldots,j-1$. 
       Then 
       \begin{equation*}
       \sigma^j[u^*](1) = \int_{-1}^1 \sigma^{j-1}[u^*] dx
                        =  - \int_{-1}^1 \sigma^{j-2}[u^*] x\,dx + \sigma^{j-1} [u^*](1)\;.
       \end{equation*} 
       The right hand side vanishes because $u^* \in \mathcal{H}^i \cap C_c^\infty(-1,1)$ and the induction assumption.
\end{itemize}
The reverse direction can be performed with an analogous induction argument.

\eqref{eq:density} follows from Remark \ref{re:hspace*} and the fact that $\mathcal{H}^i$ is closed in $L^2(-1,1)$.
\end{proof}
Using this lemma we are able to derive a characterization of the TV-seminorm via $\sigma^i$: For all $u \in L^2(-1,1)$ we have the identity:
\begin{equation*}
\begin{aligned}
  & \TV_{1}^i(u) \\
  =& \sup \set{ \int_{-1}^1 u \phi^i \,dx : \phi \in \mathcal{C}_c^\infty(-1,1)\,, \norm{\phi}_{L^\infty} \leq 1}\\
  =& \sup \left\{ \int_{-1}^1 u \phi^i \,dx : \phi^i \in \mathcal{C}_c^\infty(-1,1)\,, \phi^j(\pm 1)=0\,, j=0,1,\ldots,i-1\,, \norm{\phi}_{L^\infty} \leq 1 \right\}\\  
= & \sup \set{ \int_{-1}^1 u \psi \,dx : 
\psi \in \Psi^i \cap \mathcal{C}_c^\infty(-1,1)\,,\;
\norm{\sigma^i[\psi]}_{L^\infty} \leq 1
}\\
 \underbrace{=}_{\eqref{eq:Psi}} & \sup \set{ \int_{-1}^1 u \phi\,dx : 
\phi \in \mathcal{H}^i \cap C_c^\infty(-1,1)\,, \norm{\sigma^i(\phi)}_{L^\infty} \leq 1}\;.\\
\end{aligned}
\end{equation*}
Using \eqref{eq:density}, $\overline{\Psi^i} = \mathcal{H}^i$, and the fact that $u \to \norm{\sigma^i[u]}_{L^\infty}$ is lower semi-continuous with respect to the $L^2$-norm, it follows
that 
\begin{equation}
\label{eq:poesch}
\begin{aligned}
\TV_{\lambda_i}^i(u) &= \lambda_i \TV_{1}^i(u)\\
&= \sup \set{ \int_{-1}^1 u \phi\,dx : 
\phi \in \mathcal{H}^i\,, \norm{\sigma^i(\phi)}_{L^\infty} \leq \lambda_i}\;.
\end{aligned}
\end{equation}
From \cite[Theorem 5.1]{PoeSch08} we then get an equivalent characterization of $\TV_{\lambda_i}^i$:
\begin{equation}
\label{eq:poesch_one}
TV_1^i(u)=  \sup \set{ 
\int_{-1}^1 u \psi \, dx : 
\psi \in \mathcal{H}^i\,, \;
\dualnorm{\psi}{TV_1^i} \leq 1}\;.
\end{equation}
In an analogous ways we can rewrite the $\TGV$-functional
\begin{equation}
\label{eq:tgv_tv1_tv2}
\begin{aligned}
~ & \TGV(u)\\
= &
\sup \set{\int_{-1}^1 u \phi'' \, dx: \phi \in \mathcal{C}_c^\infty(-1,1), 
\norm{\phi'}_{L^\infty}\leq \lambda_1, 
\norm{\phi}_{L^\infty} \leq \lambda_2
}\\
=& \sup \left\{ \int_{-1}^1 u \phi \, dx: \phi \in \mathcal{H}^2 ,
\norm{\sigma^1\lre{\phi}}_{L^\infty}\leq \lambda_1, \norm{\sigma^2\lre{\phi}}_{L^\infty}\leq \lambda_2 \right\}\;.
\end{aligned}
\end{equation}

\begin{lemma}\label{le:hilfslemma}
Let $d=1$ and $\Omega = (-1,1)$. 
\begin{itemize}
\item For $u \in \mathcal{H}^2$ and $i=1,2$, we have
\begin{equation*}
\TGV(u)\leq \TV_{\lambda_i}^i(u)\;
\end{equation*}
\item For $u^* \in \mathcal{H}^2$ and $i=1,2$, we have
\begin{equation}
\label{eq:abschaetzung}
\begin{aligned}
\dualnorm{u^*}{ \TV_{\lambda_i}^i}
& \leq \norm{\sigma^i[u^*]}_{L^\infty}\,,
\dualnorm{u^*}{ \TV_{\lambda_i}^i}
&\leq 
\dualnorm{u^*}{ \TGV}\;.
\end{aligned}
\end{equation}
\end{itemize}
As a consequence 
\begin{equation}
\label{eq:TGVball}
\mathcal{B}^*_{\TGV} \subset
\mathcal{B}^*_{\TV_{\lambda_1}^1} \cap 
\mathcal{B}^*_{\TV_{\lambda_2}^2}\;.
\end{equation}

On the other hand, if $u^*$ satisfies 
\begin{equation}
 \label{eq:big_hundianer}
\norm{\sigma^i(u^*)}_{L^\infty} \leq \lambda_i,\; i=1,2, \text{ then } 
u^* \in \mathcal{B}^*_{\TGV}.
\end{equation}

Moreover,
\begin{equation}
  \TGV(u) = \sup \set{\int_\Omega u \psi\,dx : 
\psi \in \mathcal{H}^2:
\dualnorm{\psi}{\TGV} \leq 1}\;.
\end{equation}
\end{lemma}
\begin{proof}
First, we note that for every $\rho \in \mathcal{H}^2$ 
\begin{equation}
 \label{eq:polynom}
 \int_\Omega \rho p \; dx =0\,, \qquad \forall p \in \mathcal{P}^1\;.
\end{equation}
\begin{itemize}
\item From \eqref{eq:tgv_tv1_tv2} it follows that 
\begin{equation*}
\begin{aligned}
& \TGV(u) \\
= & 
\sup \left\{ \int_{-1}^1 u \phi \, dx: \phi \in \mathcal{H}^2 , \norm{\sigma^1\lre{\phi}}_{L^\infty}\leq \lambda_1, \norm{\sigma^2\lre{\phi}}_{L^\infty}\leq \lambda_2 \right\} \\
\leq& \sup \left\{ \int_{-1}^1 u \phi \, dx: \phi \in \mathcal{H}^2 ,\norm{\sigma^2\lre{\phi}}_{L^\infty}\leq \lambda_2 \right\}\\
=& \TV_{\lambda_2}^2(u)
\,,\qquad \forall u \in \mathcal{H}^2\;.
\end{aligned}
\end{equation*}
Moreover, from \eqref{eq:poesch_one} and \eqref{eq:polynom} it follows for all $u \in \mathcal{H}^2$ that 
\begin{equation*}
\begin{aligned}
& \sup \left\{ \int_{-1}^1 u \phi \, dx: \phi \in \mathcal{H}^1 , \norm{\sigma^1\lre{\phi}}_{L^\infty}\leq \lambda_1\right\} \\
= & \sup \left\{ \int_{-1}^1 u \phi \, dx: \phi \in \mathcal{H}^1 \cap \mathcal{P}^2{}^\perp ,\norm{\sigma^1\lre{\phi}}_{L^\infty}\leq \lambda_1 \right\}\\
=& \sup \left\{ \int_{-1}^1 u \phi \, dx: \phi \in \mathcal{H}^2,\norm{\sigma^1\lre{\phi}}_{L^\infty}\leq \lambda_1 \right\}\\
=& \TV_{\lambda_1}^1(u)
\,,\qquad \forall u \in \mathcal{H}^2\;.
\end{aligned}
\end{equation*}
Thus 
\begin{equation*}
\TGV(u) \leq \TV_{\lambda_1}^1(u)
\,,\qquad \forall u \in \mathcal{H}^2\;.
\end{equation*}
\item
Because $\TV_{\lambda_i}^i$ and $\TGV$ are lower semi-continuous on $L^2(-1,1)$ it follows that 
\begin{equation}\label{eq:sets}
\begin{aligned}
&\set{u \in  \mathcal{H}^i: \TV_{\lambda_i}^i(u) \leq 1}\\
= & \overline{\set{u \in  \mathcal{C}_c^\infty(-1,1)\cap \mathcal{H}^i: \TV_{\lambda_i}^i(u) \leq 1}} \\
\subset &
\overline{\set{u \in  \mathcal{C}_c^\infty(-1,1)\cap \mathcal{H}^i: \TGV(u) \leq 1}} \\
= & \set{u \in \mathcal{H}^i: \TGV(u) \leq 1}\,,\qquad \forall i=1,2.
\end{aligned}
\end{equation}
Therefore, from \eqref{eq:polynom} it follows that
\begin{align*}
& \dualnorm{u^*}{ \TV_{\lambda_i}^i}\\
= & \sup \set{\int_{\Omega} u^*u \; dx : u\in \mathcal{H}^i\,, \TV_{\lambda_i}^i(u)\leq 1}\\
= &
\sup \set{\int_{\Omega} u^*u \; dx : u\in \mathcal{H}^2\,, \TV_{\lambda_i}^i(u)\leq 1}\,,\qquad \forall u^* \in \mathcal{H}^2\;.
\end{align*}
This, together with \eqref{eq:sets}
implies that
\begin{equation}
\begin{aligned}\label{eq:qqq}
\dualnorm{u^*}{ \TV_{\lambda_i}^i}
&= \sup \set{\int_{\Omega} u^*u \; dx : u\in \mathcal{H}^2 \,, \TV_{\lambda_i}^i(u)\leq 1}\\
&\leq \sup
\set{\int_{\Omega} u^* u \; dx :  
u\in \mathcal{H}^2, \TGV\lr{u}\leq 1} \\
&=\dualnorm{u^*}{\TGV}\,, \qquad \forall u^* \in \mathcal{H}^2\;.
\end{aligned}
\end{equation}
The definition of the $*$-number shows that 
$$\dualnorm{u^*}{\TV_1^i} = \sup \set{\int_{-1}^1 u^* u\,dx: u \in \mathcal{H}^i : \TV_1^i(u)=1}\;.$$
For all $u\in \mathcal{H}^i$ satisfying $\TV_1^i(u)=1$ we have
\begin{equation*}
 \begin{aligned}
1= \TV_1^i (u)  
& = \sup \set{ \int_{-1}^1 u \phi^*\,dx\;:\;\phi^* \in \mathcal{H}^i\,,\; \norm{\sigma^i(\phi^*)}_{L^\infty} \leq 1}\\
& = \sup \set{ \int_{-1}^1 u \frac{\phi^*}{\norm{\sigma^i(\phi^*)}_{L^\infty}}\,dx\;:\;\phi^* \in \mathcal{H}^i}\,,\;.
 \end{aligned}
\end{equation*}
Choosing $\phi^*=u^*$ then gives 
\begin{equation*}
\int_{-1}^1 u u^*\,dx \leq \norm{\sigma^i(u^*)}_{L^\infty}\,, \qquad \forall u \in \mathcal{H}^i \text{ with }\TV_1^i(u)=1\;.
\end{equation*}
This shows that 
\begin{equation*}
\dualnorm{u^*}{\TV_1^i} \leq \norm{\sigma^i(u^*)}_{L^\infty}\,, \qquad \forall u^* \in \mathcal{H}^i\;.
\end{equation*}

\item
To prove \eqref{eq:big_hundianer} we use the definition of the $*$-norm:
\begin{equation*}
 \begin{aligned}
  & \dualnorm{u^*}{\TGV} \\
  = &\sup \set{ \int_{-1}^1 u^*u\,dx\;: 
  u \in \mathcal{H}^2, 
  \TGV(u) \leq 1}\,,\\
  & \TGV(u) \\
  = & \sup \set{\int_\Omega u \phi''\,dx :  \phi \in \mathcal{C}^\infty_c(-1,1),\norm{\phi'}_{L^\infty} \leq \lambda_1\,, 
  \norm{\phi''}_{L^\infty} \leq \lambda_2}\;.
 \end{aligned}
\end{equation*}
The function $\phi = \sigma^2[u^*]$ satisfies 
\begin{itemize}
\item $u^* = \phi''$,
\item and by assumption $$\norm{\phi'}_{L^\infty}  = 
\norm{\sigma^1(u^*)}_{L^\infty} \leq \lambda_1\,, \norm{\phi}_{L^\infty} = \norm{\sigma^2(u^*)}_{L^\infty} \leq \lambda_2\,,$$
\end{itemize}
If $\TGV(u) \leq 1$, it then follows that 
\begin{equation*}
  \int_{-1}^1 u u^* \; dx\leq \TGV(u) \leq 1\;.
\end{equation*}
Taking the supremum with respect to $u$ then shows that $\dualnorm{u^*}{\TGV} \leq 1$.

On the other hand according to \eqref{eq:qqq} $u^* \in \mathcal{B}^*_{\TGV}$ implies that
$$\max \set{ 
\dualnorm{u^*}{\TV^1_{\lambda_1} },
\dualnorm{u^*}{\TV^2_{\lambda_2} }
} \leq 1\;.$$ 
Because 
$$\dualnorm{u^*}{\TV^i_{\lambda_i}} = \frac{1}{\lambda_i} \norm{\sigma^i\lre{u^*}}_{L^\infty}, \qquad \forall u^* \in \mathcal{B}^*_{\TGV}\,,$$
it then follows that 
$$\norm{\sigma^i\lre{u^*}}_{L^\infty} \leq \lambda_i, \qquad \forall u^* \in \mathcal{B}^*_{\TGV}\,, \forall i=1,2\,.$$
\item 
As a consequence
\begin{align*}
~ &
  \TGV(u) \\
  =& \sup \set{\int_\Omega u \phi''\,dx : 
  \phi \in \mathcal{C}^\infty_c(\Omega), 
  \norm{\phi'}_{L^\infty} \leq \lambda_1\,, \norm{\phi''}_{L^\infty} \leq \lambda_2}\\
 =& \sup \set{\int_\Omega u \psi\,dx : 
\psi \in \mathcal{H}^2\,,
\dualnorm{\psi}{\TGV} \leq 1}\;.
\end{align*}
\end{itemize}
\end{proof}

\begin{lemma}
\label{le:regions}
For 
\begin{equation}
\label{eq:a}
\lambda_2 \geq
\norm{\sigma^2[\tvmast{1}]}_{L^\infty}\,, 
\end{equation}
we have $\tvm{1} = \utgv$ and $\TGV(\tvm{1})= \TV_{\lambda_1}^1(\tvm{1})$. 
On the other hand, if 
\begin{equation}
\label{eq:b}
\lambda_1 \geq  
\norm{\sigma^1\lre{\tvmast{2}}}_{L^\infty}\;,
\end{equation}
then $\tvm{2} = \utgv$ and $\TGV\lr{\tvm{2}}=
\TV^2_{\lambda_2 }\lr{\tvm{2}}$. 
\end{lemma} 

\begin{proof}
We only prove the first assertion. The proof of the second assertion is 
analogous, and therefore omitted.

We summarize two properties of $\tvm{1}$:
\begin{itemize}
\item By assumption $\norm{\sigma^2[\tvmast{1}]}_{L^\infty} \leq \lambda_2$.
\item Since $\tvm{1}$ minimizes \eqref{eq:l2TV} it follows from \cite{PoeSch08} that
       $\sigma_1[\tvmast{1}] \leq \lambda_1$. 
\end{itemize}
From \eqref{eq:big_hundianer} it then follows that
\begin{equation}
 \label{eq:rest2}
 \tvmast{1} \in \mathcal{B}^*_{\TGV}\;.
\end{equation}
Because $\tvm{1}$ is the minimizer of the TV-functional $\F^1$ it follows 
from Lemma \ref{le:min} and \eqref{eq:poesch} that 
\begin{equation}\label{eq:A1}
\begin{aligned}
~ & -\int_{-1}^1 (\tvm{1}- \ud)\tvm{1}\; dx \\
= & \TV_{\lambda_1}^1(\tvm{1}) \\
= & \sup \set{\int_{-1}^1 \tvm{1}v^*\,dx : v^*\in \mathcal{H}^1 \cap \mathcal{C}_c^\infty(-1,1)\,, \norm{\sigma^1(v^*)}_{L^\infty}\leq \lambda_1}\;.
\end{aligned}
\end{equation}
Moreover, since $\tvmast{1} \in \mathcal{B}^*_{\TGV}$, we have 
\begin{equation*}
 0 = \TGV^*(\tvmast{1}) = \sup_{u \in L^2(-1,1)} \set{ \int_{-1}^1 \tvmast{1}u\,dx\; - \TGV(u)}\,,
\end{equation*}
and in particular for the test function $u=-\tvm{1}$,
$$ -\int_{-1}^1 (\tvm{1} -\ud)\tvm{1}\,dx = -\int_{-1}^1 \tvmast{1} \tvm{1}\,dx\; \leq \TGV(-\tvm{1}) = \TGV(\tvm{1})\;.$$
This, together with Lemma \ref{le:hilfslemma} and \eqref{eq:A1} shows that
\begin{equation}\label{eq:B1}
-\int_{-1}^1 (\tvm{1}- u)\tvm{1}\; dx  \leq \TGV(\tvm{1}) \leq \TV_{\lambda_1}^1(\tvm{1}) = -\int (\tvm{1}- u)\tvm{1}\; dx\,,
\end{equation}
and therefore in particular 
\begin{equation}\label{eq:rest1}
-\int_{-1}^1 (\tvm{1}- u)\tvm{1}\; dx  = \TGV(\tvm{1})\;.
\end{equation}
Applying Lemma \ref{le:min} with \eqref{eq:rest2} and \eqref{eq:rest1} 
shows that 
$\tvm{1}$ also minimizes the \newline {$\TGV$-regularization} \eqref{eq:l2tgv}.
\end{proof}
\begin{definition}\label{de:Lambda}
We define  
\begin{equation*}
\Lambda:=\set{
(\lambda_1,\lambda_2): 
1 < \norm{\sigma^1\lr{\tvmast{2} }}_{L^\infty}
\text{ and }
1 < \norm{\sigma^2\lr{\tvmast{1} }}_{L^\infty}
}\;.
\end{equation*}
\end{definition}

\begin{corollary}
Let $(\lambda_1,\lambda_2) \in 
\left[
\norm{\sigma^1\lr{\ud}}_{L^\infty}
+\infty \right)
\times \left[
\norm{\sigma^2\lr{\ud}}_{L^\infty},
+\infty \right)$, then $\utgv\equiv 0$.
\end{corollary}
\begin{proof}
Because of \eqref{eq:ast_tv} and \eqref{eq:abschaetzung} we have
\begin{equation*}
 \lambda_1 \geq 
 \norm{\sigma^1(\ud)}_{L^\infty}\geq
 \dualnorm{\ud}{\TV_1^1} \text{ and }
 \lambda_2 \geq 
  \norm{\sigma^2(\ud)}_{L^\infty}\geq
  \dualnorm{\ud}{\TV_1^2} \;.
\end{equation*}
Therefore, from Example \ref{le:gnorm} and 
Lemma \ref{le:null} it follows that $\tvm{2} = \tvm{1} \equiv 0$.

Using Lemma \ref{le:regions} it follows that $\tvm{2}=\tvm{1}=\utgv$, and therefore the assertion.
\end{proof}

\begin{lemma}\label{le:charStern}
Let $\utgv^*$ be the minimizer of $u^* \to \Sim^*(u^*) + \TGV^*(-u^*)$.
\begin{enumerate}
\item \label{it:poly}  Then, on each connected component of 
\begin{equation*}
B:=\set{x: 
\abs{\sigma^1[\utgv^*](x)}<\lambda_1 \text{ and }
\abs{\sigma^2[\utgv^*](x)}<\lambda_2}\,,
\end{equation*}
$\utgv(x) \mid_{B}$ is a polynomial of maximal degree $1$. 

\item If there exists an interval $A$ such that either 
$\abs{\sigma^1[\utgv^*](x)}=\lambda_1$ for all $x\in A$ or 
$\abs{\sigma^2[\utgv^*](x)}=\lambda_2$ for all $x\in A$, 
then $\utgv=\ud$ on $A$.

\item \label{item:jumping} \textbf{Jump Condition:} 
If there exists $x_0 \in (-1,1)$ and $\epsilon > 0$ such that 
\begin{equation}
 \label{eq:s1}
 \begin{aligned}
  \sigma^1[\utgv^*](x_0) &= \lambda_1 \text{ and }\\
  \abs{ \sigma^1[\utgv^*](x)} &< \lambda_1\,, \quad \forall x \in
(x_0-\epsilon,x_0+\epsilon) \backslash \set{x_0}\;
 \end{aligned}
\end{equation}
then there exist constants $c_1$, and $d_1 \leq d_2$ such that
\begin{equation}
 \label{eq:jump}
 \utgv(x) =
\begin{cases}
c_1 x + d_1 &x \in (x_0-\epsilon,x_0)\,,\\
c_1 x + d_2 &x \in (x_0,x_0+\epsilon)\;.
\end{cases}
\end{equation}
If instead of the first condition in \eqref{eq:s1}, $\sigma^1[\utgv^*](x_0)=-\lambda_1$ holds, then 
$\utgv$ satisfies \eqref{eq:jump}, but $d_2 \leq d_1$.
\item \label{item:bending} 
\textbf{Bending Condition:} If there exists $x_0 \in 
(-1,1)$ and some $\epsilon>0$ such that 
\begin{equation}
 \label{eq:s2}
 \begin{aligned}
 \sigma^2[ \utgv^*](x_0) &= \lambda_2\,,\\
 \abs{\sigma^2[\utgv^*] (x)} &<\lambda_2\,,\quad \forall x \in (x_0-\epsilon,x_0+\epsilon) \backslash \set{x_0}\;.
 \end{aligned}
 \end{equation}
 Then 
 \begin{equation}\label{eq:bend}
 \utgv(x) =
\begin{cases}
c_1 x + d_1 &x \in (x_0-\epsilon,x_0)\\
c_2 x + d_2 &x \in (x_0,x_0+\epsilon)
\end{cases}
\end{equation}
is continuous at $x_0$, and $c_2 \leq c_1$, where the later condition we refer to negative bending.

If instead of the first condition in \eqref{eq:s2}, $\sigma^2[\utgv^*](x_0)=-\lambda_2$ holds, 
then the function is positively bending, i.e., $c_1 \leq c_2$.
\end{enumerate}
\end{lemma}

\begin{figure}
\begin{tabular}{ccc}
 (1) & (3) & (4)\\
\includegraphics[]{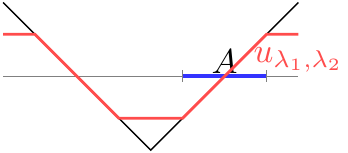}
& \includegraphics[]{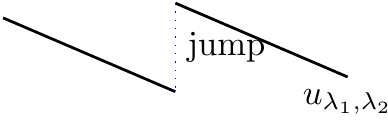}
& \includegraphics[]{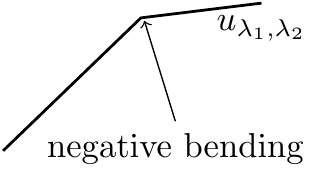}\\ 
\includegraphics[]{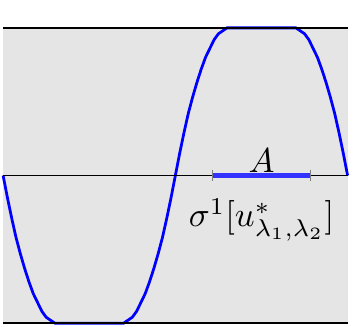}
& \includegraphics[]{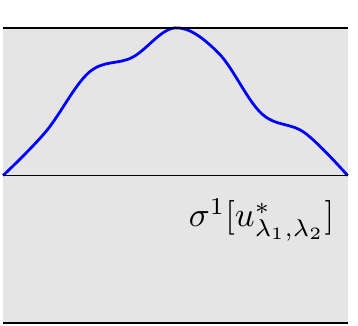}
& \includegraphics[]{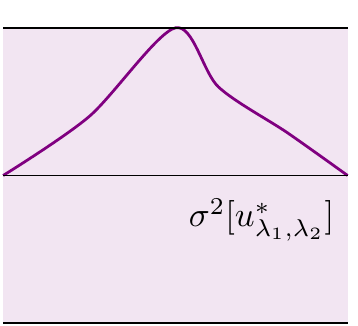}
\end{tabular}
\caption{Illustration of Lemma \ref{le:char} and \ref{le:charStern} }
\end{figure}

\begin{proof}
Recall that if $w^* \not \in \mathcal{H}^2$, then 
$w^* \not \in \dualball{\TGV}$, hence in the following, we restrict our attention to 
$w^* \in \mathcal{H}^2$. 
The Kuhn-Tucker condition
$-\utgv \in \partial \Reg^*\lr{ \utgv^*}$ guarantees that:
\begin{equation*}
 \Reg^*(v^*) - \Reg^*(\utgv^*) + \int_{-1}^1 \utgv(v^*-\utgv^*)\,dx\; \geq 0\,, \quad \forall v^* \in \mathcal{H}^2\;.
\end{equation*}
In particular, for $w^* \in \dualball{\TGV}$ we have
\begin{equation}\label{eq:kt}
\int_{-1}^1 \utgv \lr{w^*-\utgv^*}\,dx \geq 0\;.
\end{equation} 

\noindent \textbf{Item} (1):
Let $(a,b)$ be an open interval such that 
\begin{equation*} \abs{ \sigma^i[\utgv^*](x)} < \lambda_i\,, \text{ for all } i=1,2 \text{ and } x\in (a,b)\;.\end{equation*} 
Moreover, let $\phi \in \Cspace(-1,1)$ with $\supp(\phi) \subseteq (a,b)$ such that also 
\begin{equation*}\abs{\sigma_i[\utgv^*](x) + \phi^{(2-i)}(x)} < \lambda_i\,, \text{ for all }i=1,2 \text{ and } x\in (a,b)\;.\end{equation*} 
Then, 
\begin{equation*}w^*:= \utgv^* + \phi'' \in \dualball{\TGV}\end{equation*}
and therefore, it follows from \eqref{eq:kt} that
\begin{equation*}
-\int_a^b \phi'' \utgv\,dx\; \leq 0\,, \quad \forall \phi \in \Cspace(-1,1) 
\text{ with } \supp(\phi) \subseteq (a,b)\;.
\end{equation*}
Hence, $\utgv$ is a polynomial of order one in the interval $(a,b)$.

\noindent \textbf{Item} (2):
\begin{itemize}
\item \noindent i) Assume that 
      \begin{equation*}\sigma^1[\utgv^*](x) = \lambda_1\,, \quad \forall x \in (a,b)\;.\end{equation*}
      Then, 
      \begin{equation*}\ud(x)-\utgv(x) = \utgv^*(x) = \left(\sigma^1[\utgv^*]\right)'(x) = 0\,, \quad \forall x \in (a,b)\,,\end{equation*} 
      and therefore $\ud(x)=\utgv(x)$ in $(a,b)$. 
      
      \noindent ii) Assume that 
      \begin{equation*}\abs{\sigma^2[\utgv^*](x)} = \lambda_2\,, \quad \forall x \in (a,b)\;.\end{equation*} 
      From this it follows that 
      \begin{equation*} \begin{aligned}
       0 &= \left(\sigma^2[\utgv^*]\right)''(x) \\
         &= \left(\sigma^1[\utgv^*]\right)'(x) \\
         &= \utgv(x) - \ud(x)\,, \quad \forall x \in (a,b)\,,
       \end{aligned}
      \end{equation*}
      and therefore $\ud(x)=\utgv(x)$ in $(a,b)$.
  \end{itemize}

\noindent \textbf{Item} (4): Item \ref{item:bending} is based on the Assumption that there exists 
      $\epsilon > 0$ and $x_0 \in (-1,1)$  such that 
      \begin{equation*} \sigma ^2[\utgv^*](x_0)=\lambda_2 \text{ and } \sigma^2[\utgv^*](x_0 \pm y)<\lambda_2\,, \quad \forall y \in (0,\epsilon)\;.\end{equation*}
      Then, from Item \ref{it:poly} it follows that $\utgv$ is piecewise affine linear in $(x_0-\epsilon,x_0+\epsilon)$. To be precise, there exists coefficients 
      $c_1,d_1,c_2,d_2$ such that 
      \begin{equation}
      \label{eq:pwpol}
      \utgv(x) = \left\{ \begin{array}{c}
                           c_1 x + d_1 \quad \forall x \in (x_0-\epsilon,x_0)\,,\\
                           c_2 x + d_2 \quad \forall x\in (x_0,x_0+\epsilon)\;.
                          \end{array}
                          \right.
      \end{equation}
      
      We prove the assertion of Item \ref{item:bending} in two steps. 
\begin{enumerate}
\item Firstly we show that the coefficients of the piecewise polynomial satisfy 
$c_1\geq c_2$.
\item Secondly we show that $\utgv$ is continuous at $x_0$, such that we can conclude that it is
bending at $x_0$.
\end{enumerate}

\begin{itemize}
 \item[a)] To prove the first item, $c_1 \geq c_2$, we use some 
       $w^* \in \dualball{\TGV}$  (see Figure \ref{fig:w1}) satisfying 
       \begin{equation}
       \label{eq:wast_1a0}
         \sigma^1[w^*](x_0)=0\,, 
       \end{equation}
       \begin{equation}
       \label{eq:wast_1a}
       \begin{aligned}
       \sigma^1[w^*](x) &= \sigma^1[\utgv^*](x)\,, \quad \forall x \not \in (x_0-\epsilon,x_0+\epsilon)\,,\\
       \sigma^1[w^*](x) &\leq \sigma^1[\utgv^*](x)\,, \quad \forall x \in (x_0-\epsilon,x_0)\,,\\
       \sigma^1[w^*](x) &\geq \sigma^1[\utgv^*](x)\,, \quad \forall x \in (x_0,x_0+\epsilon)\,,\\
       \mu &:=-\int_{x_0 - \epsilon}^{x_0} 
(\sigma^1[w^*]-\sigma^1[\utgv^*])\,dx \\ &=
       \int_{x_0}^{x_0 + \epsilon}(\sigma^1[w^*]-\sigma^1[\utgv^*])\,dx\,,
       \end{aligned}
       \end{equation}
       and
       \begin{equation}
       \label{eq:wast_1b}
       \begin{aligned}       
       \sigma^2[w^*](x) &< \sigma^2[\utgv^*](x)\,, \quad \forall x \in (x_0-\epsilon,x_0+\epsilon) \backslash \set{x_0}\,, \\
       0 & <\sigma^2[\utgv^*](x_0)-\sigma^2[w^*](x_0)< \lambda_2\:.
       \end{aligned}
       \end{equation}
       With such a function $w^*$ it follows from \eqref{eq:kt} that
\begin{align*}
0 & \leq   \int_{-1}^1 \utgv \lr{w^*-\utgv^*}\,dx \\
  &   =  - \int_{x_0-\epsilon}^{x_0+\epsilon} \utgv' (\sigma^1[w^*] - \sigma^1[\utgv^*])\,dx \\
  &   =
c_1 \underbrace{ \int_{x_0-\epsilon}^{x_0} (\sigma^1[w^*]
-\sigma^1[\utgv^*])\,dx}_{=-\mu}  \\
& \qquad +c_2 
\underbrace{ 
\int_{x_0}^{x_0 + \epsilon} (\sigma^1[w^*] -\sigma^1[\utgv^*])\,dx}_{=\mu}\\
& = \mu(c_2-c_1)\,,
\end{align*}
which shows that $c_1\leq c_2$ since $\mu \geq 0$. 
\begin{figure}
 \begin{center}
\includegraphics{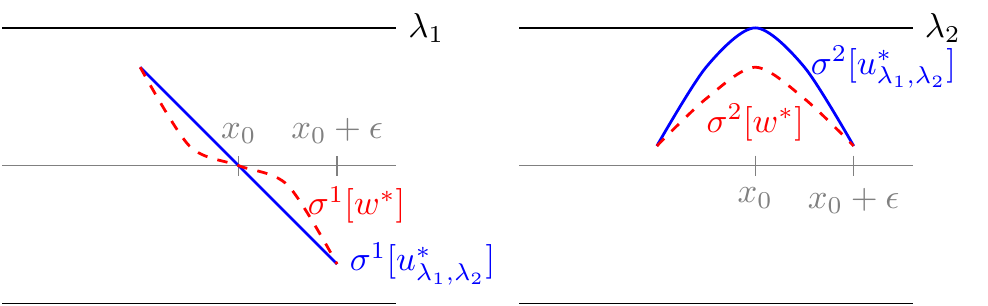}
\end{center}
\caption{
The figure shows the construction of $w^*$ satisfying \eqref{eq:wast_1a} and \eqref{eq:wast_1b}.
}\label{fig:w1}
\end{figure}

\item [b)] To prove the continuity of $\utgv$ we use a function $w^* \in 
\dualball{\TGV}$  which satisfies:
\begin{subequations}
\label{eq:w3}
\begin{align}
 \sigma^2[w^*](x)&=\sigma^2[\utgv^*](x)\,,\quad \forall x\not \in 
(x_0-\epsilon,x_0 + \epsilon) \label{eq:w3a}\\
\sigma^2[w^*](x)&\not = \sigma^2[\utgv^*](x)\,,\quad \text{for a. a.} x \in 
(x_0-\epsilon,x_0 + \epsilon) \label{eq:w3b}\\
\sigma^1[w^*-u^*](x_0)&=a \not = 0\,,  \label{eq:w3c}\\
\int_{x_0-\epsilon}^{x_0} \lr{ \sigma^1[w^*]-\sigma^1[\utgv^*]} \,dx&=
- \int_{x_0}^{x_0+\epsilon} \lr{\sigma^1[w^*]-\sigma^1[\utgv^*] }\,dx \;.
\end{align}
\end{subequations}
Such a function is represented in Figure \ref{fig:w2}.
\begin{figure}
\begin{center}
\includegraphics{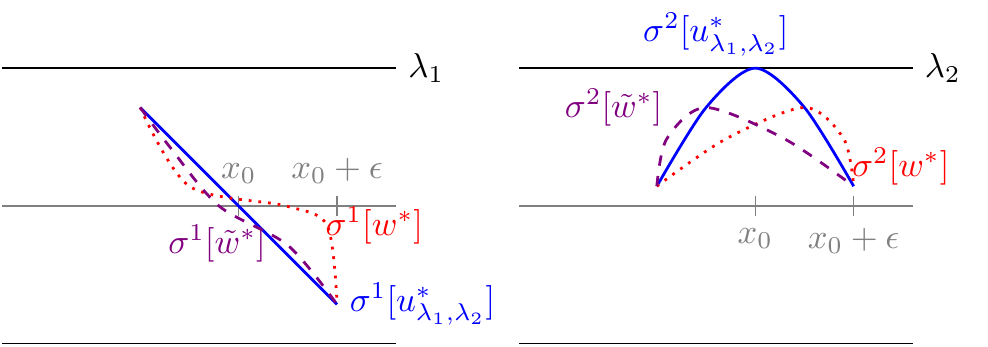}
\end{center}
\caption{$w^*$ satisfying \eqref{eq:w3}.}
\label{fig:w2}
\end{figure}
With such a function $w^*$ it follows from \eqref{eq:kt}, \eqref{eq:pwpol}, and integration by parts, that 
\begin{equation}
\label{eq:ungl1}
\begin{aligned}
 0 \leq & \int_{-1}^1 \utgv \lr{w^*-\utgv^*}\,dx \\
 =& \underbrace{\sigma^1[w^*-\utgv^*](x_0)}_{=a}
\lr{-c_1 x_0
-d_1 + c_2 x_0 +d_2} \\
&- 
c_1 \int_{x_0-\epsilon}^{x_0} (\sigma^1[w^*]-\sigma^1[\utgv^*])\,dx\\ 
&-
c_2 \int_{x_0}^{x_0 + \epsilon} (\sigma^1[w^*]-\sigma^1[\utgv^*])\,dx\;.
\end{aligned}
\end{equation}
Choosing $\tilde{w}^*$ such that $\sigma^1[\tilde{w}^*-\utgv^*](x_0)=-a$, but 
otherwise satisfying the same properties as $w^*$, that are \eqref{eq:w3a} 
and \eqref{eq:w3b}, then we obtain 
\begin{equation}
\label{eq:ungl2}
 \begin{aligned}
  0 &\leq  \int_{-1}^1 \utgv \lr{\tilde{w}^*-\utgv^*}\,dx \\
  &=
\underbrace{\sigma^1[\tilde{w}^*-\utgv^*](x_0)}_{=-a}
\lr{-c_1 x_0
-d_1 + c_2 x_0 +d_2}  \\
& - 
c_1 \int_{x_0-\epsilon}^{x_0} (\sigma^1[\tilde{w}^*]  -\sigma^1[\utgv^*])\,dx \\
& -
c_2 \int_{x_0}^{x_0 + \epsilon} (\sigma^1[\tilde{w}^*]   -\sigma^1[\utgv^*])\,dx\;.
 \end{aligned}
\end{equation}
Combining \eqref{eq:ungl2} and \eqref{eq:ungl1} finally shows
\begin{equation*}
\begin{aligned}
a \lr{-c_1 x_0 -d_1 + c_2 x_0 +d_2} &\leq 0\\
& \leq a \lr{-c_1 x_0 -d_1 + c_2 x_0 +d_2}\,,
\end{aligned}
\end{equation*}
such that we conclude that $-c_1 x_0
-d_1 + c_2 x_0 +d_2 =0$, which shows that $\utgv$ is continuous at $x_0$.
\end{itemize}

\noindent \textbf{Item} (3): Assume that $\utgv$ is as in \eqref{eq:bend}.
In the case where 
\begin{equation*}
\begin{aligned}
\sigma^1[\utgv^*](x_0) &=\lambda_1 \text{ and } \\
\sigma^1[\utgv^*](y) &<\lambda_1\,, \quad \forall y \in (x_0-\epsilon,x_0+\epsilon)\,,\setminus \set{x_0}
\end{aligned}
\end{equation*}
we select some $w^* \in \mathcal{H}^2$  such that 
 \begin{equation*}
 \utgv^*(x)=w^*(x)\,, \quad \forall x\not \in (x_0-\epsilon,x_0 +\epsilon)\,,
 \end{equation*}
 \begin{equation}\label{eq:psi1}
 \utgv^*(x)<w^*(x)\,, \quad \forall x \in (x_0-\epsilon,x_0)\,, 
 \end{equation}
 \begin{equation}\label{eq:psi2}
 \utgv^*(x) >w^*(x)\,, \quad \forall x \in (x_0,x_0+\epsilon)\,,
 \end{equation}
 and
 \begin{equation}\label{eq:psi3}
 \sigma^1[\utgv^*](x_0)=\sigma^1[w^*](x_0)\;.
 \end{equation}
Defining
\begin{align*}
 a:=\int_{x_0-\epsilon}^{x_0}\sigma^1[w^*-\utgv^*]\,dx
 \underset{\eqref{eq:psi3}}{=}- \int_{x_0}^{x_0+\epsilon}\sigma^1[w^*-\utgv^*]\,dx\,,
\end{align*}
\eqref{eq:kt} can be rewritten as 
\begin{align*}
0 \geq & \int_{-1}^1 \utgv \lr{ w^*-\utgv^*}\,dx \\
= & - (c_1 x_0+d_1){\sigma^1[w^*-u^*](x_0)} - c_1 a\\
  & + (c_2 x_0+d_2){\sigma^1[w^*-u^*](x_0)} + c_2 a\\
= & a (c_2-c_1) \;.
\end{align*}
Now replacing conditions \eqref{eq:psi1},\eqref{eq:psi2}, by 
\begin{align*}
 \utgv^*(x)>w^*(x)\,, \quad \forall x \in (x_0-\epsilon,x_0)\,, \\
 \utgv^*(x)<w^*(x)\,, \quad \forall x \in (x_0,x_0+\epsilon)\,,
\end{align*}
and again using \eqref{eq:kt} we also obtain
$a (c_1-c_2) \leq 0$. Thus $c_1=c_2$. 

\noindent (jumping down when $\sigma^1[\utgv^*] = +\lambda_1$). 
Using the same arguments as in previous items, we can also proof that 
$d_1\geq d_2$.
\end{proof}

\begin{lemma}\label{le:char}
 Let $\utgv$ be the minimizer of $u\rightarrow \Sim(u)+ \TGV(u)$. 
\begin{enumerate} 
 \item\label{it:char1} If there exists $x_0 \in (-1,1)$, such that 
$\utgv$ is as in \eqref{eq:jump}
(jumping up $d_1 \leq d_2$), 
then $\sigma^1\lre{ \utgv^* } (x_0) = \lambda_1$. 
\item \label{it:char2} If there exists $x_0 \in (-1,1)$, such that 
$\utgv$ is as in \eqref{eq:bend} and $c_2 \leq c_1$ (negative bending), 
then $\sigma^2\lre{ \utgv^* } (x_0) = \lambda_2$.
\item \label{it:char3} If there exists an interval $A$ such that 
$\utgv(x)=\ud(x)$ for 
$x\in A$, then one of the two statements holds
\begin{enumerate}
\item\label{it:a} $\abs{\sigma^1\lre{\utgv^*}(x)} = \lambda_1$ for $x\in A$, or
\item\label{it:b} $\abs{\sigma^2\lre{\utgv^*}(x)} = \lambda_2$ and 
$\sigma^1\lre{\utgv^*}(x)=0$ for $x \in A$.
\end{enumerate}
\end{enumerate}
\end{lemma}
\begin{proof}
Recall that if $\utgv$ is different from $\ud$, then $\utgv$ is a polynomial 
(piecewise). 
Set $I:=(x_0-\epsilon,x_0 +\epsilon)$.
\begin{enumerate}
 \item Now assume that $\utgv$ is as in \eqref{eq:jump} and 
 $\sigma^1\lre{\utgv^*}(x)< \lambda_1$ for $x\in I$. 
 Then we can find $w^* \in \dualball{1}{\lambda_1} \cap 
\dualball{2}{\lambda_2}$ 
 such that 
\begin{align*}
 w^*(x) &= \utgv^*(x)  &\text{for }x \in (-1,1) \setminus I,\\
 \sigma^1\lre{w^*}(x_0)&>\sigma^1 \lre{\utgv^*}(x_0), \\
 \sigma^2\lre{w^*}(x)&=\sigma^2 \lre{\utgv^*}(x) &\text{for }x \in (-1,1) 
\setminus I\;.
\end{align*}
The last condition implies that 
\begin{equation*}
 \int_{x_0-\epsilon}^{x_0+\epsilon} x \lr{ \utgv^*-w^*} dx =0
\end{equation*}
such that
\begin{align*}
\int_{-1}^{1} \utgv (\utgv^*-w^*) dx
&= 
d_1 \int_{x_0-\epsilon}^{x_0} \lr{ \utgv^*-w^*}\,dx+
d_2 \int_{x_0}^{x_0 + \epsilon} \lr{ \utgv^*-w^*}\,dx\\
&=\underbrace{(d_2 - d_1)}_{\geq 0}
  \underbrace{ \lr{  \sigma^1\lre{w^*}(x_0) - 
   \sigma^1\lre{\utgv^*}(x_0)}}_{>0}>0\;.
\end{align*}
Now this, together with \eqref{eq:optiTGV}, would give
\begin{equation*}
 \TGV(\utgv) = - \int_{-1}^{1} \utgv \utgv^* dx < -\int_{-1}^{1} \utgv w^*dx\;,
\end{equation*}
which contradicts the definition of $\TGV$ as the supremum of such integrals.
Hence $\sigma^1\lre{\utgv^*}$ must be maximal at $x_0$.
\item 
Set $I:=(x_0-\epsilon,x_0 +\epsilon)$ and assume that 
$\utgv$ is as in \eqref{eq:bend} and 
\newline
$\sigma^2\lre{\utgv^*}(x)< \lambda_2$ for $x\in I$. 
Then we can find $w^* \in \dualball{\TGV}$ 
such that 
\begin{align*}
 w^*(x) &= \utgv^*(x)  &\text{for }x \in (-1,1) \setminus I\\
 \sigma^2\lre{w^*}(x_0)&>\sigma^2 \lre{\utgv^*}(x_0) \\
 \sigma^1\lre{w^*}(x)&=\sigma^1 \lre{\utgv^*}(x) &\text{for }x \in (-1,1) 
\setminus I\;.
\end{align*}
The last condition and the continuity of $\utgv$ at $x_0$ imply that 
\begin{align*}
 \int_{x_0-\epsilon}^{x_0+\epsilon} &\utgv \lr{ \utgv^*-w^*} dx
= - \int_{x_0-\epsilon}^{x_0+\epsilon} \utgv' \sigma^1\lre{ \utgv^*-w^*} dx\\
&= - c_1 \int_{x_0-\epsilon}^{x_0}  \sigma^1\lre{ \utgv^*-w^*} dx
   - c_2 \int_{\epsilon}^{x_0+\epsilon}  \sigma^1\lre{ \utgv^*-w^*} dx
\end{align*}
such that
\begin{align*}
\int_{-1}^{1} \utgv \lr{\utgv^*-w^*} dx
=\underbrace{ (c_2 - c_1)}_{\leq 0}
 \underbrace{ \lr{  \sigma^2\lre{\utgv^*-w^*}(x_0)} }_{<0}>0\;.
\end{align*}
Now this, together with \eqref{eq:optiTGV}, would give
\begin{equation*}
 \TGV(\utgv) = - \int_{-1}^{1} \utgv \utgv^* dx< -\int_{-1}^{1} \utgv w^* dx\;,
\end{equation*}
which contradicts the definition of $\TGV$ as a supremum.
Hence \newline $\sigma^1\lre{\utgv^*}$ must be maximal at $x_0$.

\item The proof is analog to (1),(2).

\end{enumerate}
\end{proof}

\section{Example 1}
\label{se:ex1}
In the following we calculate the specific minimizers of $\TV, \TV^2$ and $\TGV$-minimization
for the test data, 
\begin{equation}
\label{eq:ud_ex1}
 \begin{aligned}
  \ud: (-1,1) &\to \R\\
  x &\mapsto \abs{x}-\frac{1}{2}\;.
 \end{aligned}
\end{equation}
In this case we have
\begin{equation*}
\lambda_1 \dualnorm{\ud}{\TV_{\lambda_1}^1} =
\dualnorm{\ud}{\TV_1^1} =\norm{\sigma^1[\ud]}_{L^\infty}
=\int_{-1}^{-\frac{1}{2}}\ud \,dx = \frac{1}{8}\end{equation*}
and   
\begin{equation*}
\lambda_2 \dualnorm{\ud}{\TV_{\lambda_2}^2} =  \dualnorm{\ud}{\TV_1^2}=\norm{\sigma^2[\ud]}_{L^\infty}=
\abs {\sigma^2[\ud](0)} =
\frac{1}{12} =\frac{2}{3}\dualnorm{\ud}{\TV_1^1}\;.\end{equation*} 
\begin{figure}
\includegraphics[scale=0.95]{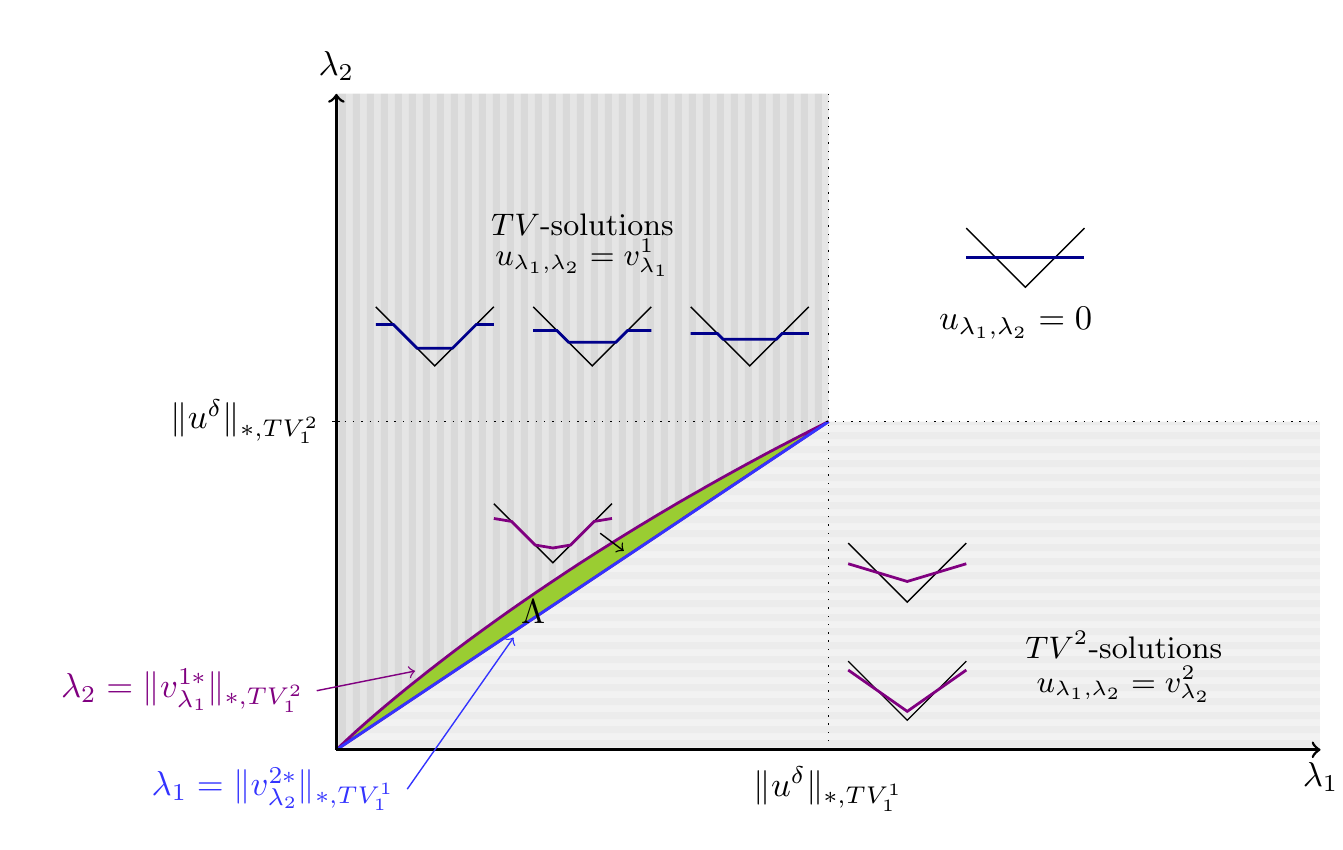}
\caption{
This $(\lambda_1,\lambda_2)$-diagram show regions where $\G$-minimizers are
different or equal to $F^i$-minimizers. 
In the region with the horizontal lines we have $\utgv=\tvm{2}$, that 
is, the $\TGV$-minimizer equals the $\TV^2$-minimizer.
In the green region where $(\lambda_1,\lambda_2) \in \Lambda$, $\TGV$-minimizers are different from $\TV^1,\TV^2$ minimizers, respectively. 
}
\label{fig:region1}
\end{figure}

\subsection{$\TV$-minimizer} 
Using the same methods as in \cite{PoeSch08}, we find that for given data \eqref{eq:ud_ex1}, the minimizer of the $\TV$-functional $\F^1$ is given by 
\begin{align*} 
\tvm{1} = 
\begin{cases} 
\sqrt{2 \lambda_1}-\frac{1}{2}&\abs{x}\leq \sqrt{2 \lambda_1}\\ 
\ud(x)&\sqrt{2 \lambda_1}<\abs{x}\leq 1-\sqrt{2 \lambda_1}\\ 
\frac{1}{2}-\sqrt{2 \lambda_1}&1-\sqrt{2 \lambda_1}< \abs{x}\leq 1 
\end{cases} 
\end{align*} (see Figure \ref{fig:exabs} (right)). 

The function $\tvm{1}$ and its dual $\tvmast{1}$ satisfy the following properties:
\begin{enumerate}
\item $\TV\lr{ \tvm{1}}=2-4\sqrt{2 \lambda_1}$,
\item $\dualnorm{\tvmast{1}}{\TV_1^1} = \sigma^1\lre{\tvmast{1}}\lr{\frac{1}{2}} =
\lambda_1$, and 
\item $\dualnorm{\tvmast{1}}{\TV_1^2}  = \abs{\sigma^2\lre{\tvmast{1}}(0)} =
\lambda_1\lr{1-\frac{2}{3}\sqrt{2\lambda_1}}.$
\end{enumerate}
Hence, from Lemma \ref{le:regions} it follows that, as long as
\begin{align}\label{eq:bound1a}
\lambda_2\geq \dualnorm{\tvmast{1}}{\TV_1^2} = \lambda_1\lr{1-\frac{2}{3}\sqrt{2 \lambda_1}},
\end{align} 
the $\TV^1$-minimizer is also the $\TGV$-minimizer and 
\begin{equation*} 
\lambda_1 \TV^1_1(\tvm{1}) = \TV_{\lambda_1}^1(\tvm{1}) = \TGV(\tvm{1})= \lambda_1 (2-4\sqrt{2\lambda_1})\,.
\end{equation*}

\subsection{$\TV^2$-minimizers} 
For $\ud$ from \eqref{eq:ud_ex1} the minimizer of $\F^2$ is given by 
\begin{equation*}\tvm{2}=\lr{ 1-\frac{1}{ \dualnorm{\ud}{\TV_1^2}}\lambda_2}^+ \ud \,,\end{equation*}
where 
\begin{equation*} f^+(x) = \max \set{f(x),0}\,.\end{equation*}
Using Lemma \ref{le:regions} it follows that for
\begin{align}\label{eq:bound2a}
\lambda_1\geq \dualnorm{\tvmast{2}}{\TV_1^1} =
\lr{1-\frac{1}{ \dualnorm{\ud}{\TV_1^2}}\lambda_2} \dualnorm{\ud}{\TV_1^1}
\end{align}
the $\TV^2$-minimizer, i.e., the minimizer of $\F^2$, is also a minimizer 
of the $\TGV$-functional $\G$.
In Figure \ref{fig:region1} we illustrate the
$(\lambda_1,\lambda_2)$-region where the minimizers of $\G$ are equal to minimizers of
$\F^2$.

\subsection{$\TGV$-minimizer} 
\newcommand{\clambda}{c_{\lambda_1,\lambda_2}}
\newcommand{\dlambda}{d_{\lambda_1,\lambda_2}}
\newcommand{\wtgv}{w_{\lambda_1,\lambda_2}}

Firstly, we calculate the set $\Lambda$ (cf. Definition \ref{de:Lambda}) for which the $\TGV$-minimizer is different from 
the $\TV^i$-minimizers, respectively. 
For this particular data $\ud$ this means that for 
\begin{equation*}\lambda_2 \notin \Lambda_2:=\left[
\frac{1}{12}-\frac{2}{3}\lambda_1,
\lr{1-\frac{2}{3}\sqrt{2 \lambda_1}}\lambda_1 \right]\,,\end{equation*}
the minimizer of the $\TGV$-functional $\G$ equals a minimizer of a
$\TV^i$-functional $\F^i$, for some $i=1,2$.

Let now $\lambda_2 \in \Lambda_2$, which is the only case for which we can expect that the 
$\TGV$-minimizer is different to $\TV^i$-minimizers.

We introduce the two-parametric set of functions $W$, consisting
of all functions of the form,
\begin{align}
\label{eq:wtgv}
 w(x,c,d):=  
\begin{cases}  
d \abs{x}+c(1-d)-\frac{1}{2}& \abs{x}\leq c \\  
\ud  &c< \abs{x}\leq 1-c \\  
d \abs{x}+c(d-1)-d+\frac{1}{2}& \abs{x} > 1 - c 
\end{cases}\;,
\end{align}  
where $c\in \lre{0,\frac{1}{2}}$ and $d \in \lre{0,1}$. 
Note that 
\begin{itemize}
\item $w\lr{x,c,d}$ is continuous,
\item $w\lr{x,0,1}=w\lr{x,\frac{1}{2},1}=\ud(x)$ and $w\lr{x,\frac{1}{2},0}=0$,
\item $w\lr{x,c,0}=\tvm{1}(x)$ for $\lambda_1=\frac{1}{2}c^2$, 
\item $w\lr{x,\frac{1}{2},d}=\tvm{2}(x)$ for 
$\lambda_2=\lr{1-d} \dualnorm{\ud}{\TV_1^2}$.
\end{itemize}

\begin{center}
\includegraphics[]{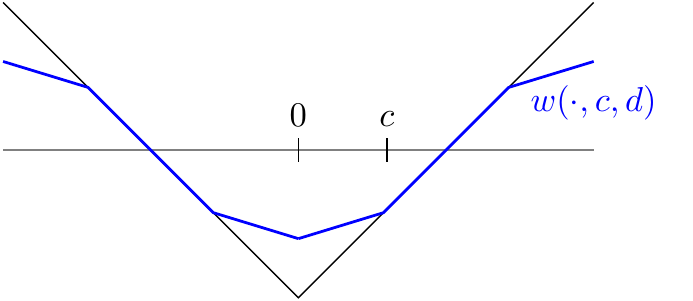}
\end{center}

Assuming that $\wtgv:=w(\cdot ,\clambda,\dlambda)$ minimizes $\G$, Lemma \ref{le:char} provides necessary criteria for optimality of the 
parameters $\clambda$ and $\dlambda$, which are derived in the following. Then, in Theorem \ref{th:wu} below, we prove that $\wtgv$ 
in fact minimizes $\G$.

Assuming that $\wtgv$ is a minimizer of $\mathcal{G}$ it follows from Lemma \ref{le:char} that:
\begin{itemize}
\item For $(\lambda_1,\lambda_2)$, such that $\dlambda>0$, $\wtgv$ bends at 
$x=0$.
In Remark \ref{re:41}, we calculate the coefficients such that $\dlambda=0$.
Lemma \ref{le:char} item \ref{it:char2} states that then 
$\wtgv^*=\wtgv-\ud$ satisfies 
\begin{align}
\label{eq:i}
\sigma^2\lre{\wtgv^*} (0) = -\lambda_2\;.
\end{align}
\item Lemma \ref{le:char} item \ref{it:a} states that since
$\wtgv(x) = \ud(x)$,  \newline $x\in (-1+\clambda,-\clambda)$,
we have
\begin{align}
 \lambda_1&=\abs{\sigma^1\lre{\wtgv^*}(\clambda)}=
\abs{\sigma^1\lre{\wtgv^*}(1-\clambda)} \;.
\label{eq:ii}
\end{align}
(Item \ref{it:b} cannot occur in this case, because  
$\sigma^1[\wtgv^*](-1+\clambda )\not=0$, for any $\dlambda \not= 1$)
\end{itemize}
Using a Computer Algebra system, we solve \eqref{eq:i}-\eqref{eq:ii} and obtain
\begin{align}  \label{eq:coefficients}
\clambda&=  \frac{3\lr{\lambda_1-\lambda_2}}{2 \lambda_1}\;,  
&\dlambda&= 1-\frac{8}{9}
\frac{\lambda_1^2}{(\lambda_2-\lambda_1)^2}\lambda_1\;.
\end{align}

\begin{remark}\label{re:41}
We want to see what happens for the special case when 
$(\lambda_1,\lambda_2) \in \partial \Lambda$, that is we consider the two sets 
of parameters: 
\begin{align*}
&\set{(\lambda_1,\dualnorm{\tvmast{1}}{\TV_1^2}), \lambda_1 \in 
(0,\dualnorm{\ud}{\TV_1^1})} \;,\\
&\set{(\dualnorm{\tvmast{2}}{\TV_1^1}, \lambda_2), \lambda_2 \in 
(0, \dualnorm{\ud}{\TV_1^2})}\;.
\end{align*}

\begin{itemize}
\item In the case $\lambda_2=\lambda_1\lr{1 + \frac{2}{3}\sqrt{2 \lambda_1}}$
(this is the case  where $\dualnorm{\tvmast{1}}{\TV_{\lambda_2}^2} = 1$, see Lemma 
\ref{le:regions}), \eqref{eq:coefficients} gives
$\dlambda=0$ and $\clambda=\sqrt{2\lambda_1}$.
One can see that then $\wtgv$ is either piecewise constant or equal to $\ud$ on
$(-1+\clambda,-\clambda) \cup (\clambda,1-\clambda)$. 
We see that for this particular choice of 
$(\lambda_1,\lambda_2)$ we have $\wtgv=\tvm{1}$, hence $\wtgv$ also minimizes 
$\F^1$.

\item For $\lambda_1 = \frac{3}{2} \lambda_2$ (this is the case where 
$\dualnorm{\tvmast{2}}{\TV_1^1} = \lambda_1$, see Lemma \ref{le:regions}), we have
$\clambda=\frac{1}{2}$ and $\dlambda =
\lr{1-\frac{\lambda_2}{ \dualnorm{\ud}{\TV_1^2}}}$. 
We see that $\wtgv=\tvm{2}$. 
\end{itemize}
\end{remark}

\begin{theorem}\label{th:wu}
For $(\lambda_1,\lambda_2) \in \Lambda$ and $\clambda$, $\dlambda$ satisfying 
\eqref{eq:coefficients}, 
$\utgv=\wtgv$.
\end{theorem}

In order to proof the theorem in a compact way, we need the following remark:
\begin{remark}
In the next two items, we only rewrite $\wtgv$ as a linear combination of 
minimizers of $\F^i$, where we have to replace $\lambda_i$ by a different
parameter $\mu_i$ depending on $\lambda_1,\lambda_2$.

\begin{itemize}
\item 
For given 
${\mu_1} \in \left[ 0,\dualnorm{\ud}{\TV_1^1} \right], 
 {\mu_2} \in \left[0, \dualnorm{\ud}{\TV_1^2} \right]$   
set 
\begin{align}\label{eq:l1l2}
\lambda_1=\frac{{\mu_1} {\mu_2}}{ \dualnorm{\ud}{\TV_1^2}}
\qquad \text{ and }\qquad
\lambda_2 = \frac{{\mu_2}}{ \dualnorm{\ud}{\TV_1^2}} 
\dualnorm{\tvmm{1}}{\TV_1^2} =
12 \, {\mu_2} \lr{{\mu_1}-\frac{2}{3}\sqrt{2{\mu_1}^3}}.
\end{align} 
Comparing the coefficients of the piecewise terms of $\wtgv$, 
we see that for $(\lambda_1,\lambda_2) \in \Lambda$ we 
can write
\begin{equation}\label{eq:sum1}
 \wtgv=\frac{{\mu_2}}{ \dualnorm{\ud}{\TV_1^2}} \tvmm{1} +
\tvmm{2}\;
\end{equation}

\item 
On the other hand, for $(\lambda_1,\lambda_2) \in \Lambda$ given, we calculate 
${\mu_1}, {\mu_2}$ by
\begin{align*}
{\mu_2} =\frac{2}{27}\frac{\lambda_1^2}{(\lambda_1-\lambda_2)^2}, \quad
{\mu_1} =\frac{9}{8}\frac{(\lambda_1-\lambda_2)^2}{\lambda_1^2} 
\end{align*}
and express $\wtgv$ by \eqref{eq:sum1}.
\end{itemize}
\end{remark}
\begin{proof}
Using the triangle-inequality and the estimate $\TGV(u)\leq \lambda_1 TV^i(u)$, 
we obtain 
\begin{equation}\label{eq:eq1b}
\begin{aligned}
\TGV(\wtgv) 
&\lequnder{eq:sum1} \TGV \lr{\frac{{\mu_2}}{ \dualnorm{\ud}{\TV_1^2}} \tvmm{1}} 
+\TGV\lr{\tvmm{2}}\\
&\leq \lambda_1 \frac{{\mu_1}}{\dualnorm{\ud}{\TV_1^1}} TV^1_1\lr{\tvmm{1}}
+
\lambda_2   TV^2_1\lr{\tvmm{2}}
\;.
\end{aligned}
\end{equation}
Due to the definition of $\TGV$ and the choice of the
parameters
$\lambda_1,\lambda_2$, we have that $\wtgv^*:=\wtgv-\ud \in
\dualball{\TGV}$, such that
\begin{equation}\label{eq:eq1a}
-\int_{-1}^1 \wtgv (\wtgv-\ud)\,dx \leq \TGV(\wtgv) \;.
\end{equation}
In order to simplify the left side, we calculate
\begin{align*}
 -\int_{-1}^1 \tvmm{1} (\wtgv-\ud)  \,dx&=
 -\frac{{\mu_1}}{\dualnorm{\ud}{\TV_1^1}} \int_{-1}^1 \tvmm{1} (\tvmm{1}-\ud)  \,dx\\
 &= 
 \frac{{\mu_1}}{\dualnorm{\ud}{\TV_1^1}} \mu_1 TV^1\lr{\tvmm{1}}\;
\end{align*}
and
\begin{align*}
& -\int_{-1}^1 \tvmm{2} (\wtgv-\ud) \,dx\\
= & 
- \lr{   1- \frac{\mu_2}{  \dualnorm{\ud}{\TV_1^2} } } 
\frac{\mu_2}{\dualnorm{\ud}{\TV_1^2}}
\int_{-1}^1 \ud \lr{\tvmm{1}- \ud} \,dx\\
= &-  \lr{1-\frac{{\mu_2}}{ \dualnorm{\ud}{\TV_1^2}}}  
\frac{\mu_2}{\dualnorm{\ud}{\TV_1^2}}
 \lr{
 \int_{-1}^1 \abs{x} \lr{\tvmm{1}- \ud} \,dx
 - \frac{1}{2} 
\underbrace{ \int_{-1}^1  \lr{\tvmm{1}- \ud}\,dx}_{
\sigma^1\lre{\tvmm{}\!\!\!^*}(1)=0}}\\
= &-  \underbrace{ 
 \lr{1-\frac{{\mu_2}}{ \dualnorm{\ud}{\TV_1^2}}}  
 }_{TV^2\lr{\tvmm{2}} \frac{1}{2}}
 \frac{\mu_2}{\dualnorm{\ud}{\TV_1^2}}
 \underbrace{    \lr{\int_{-1}^{0} \sigma^1\lre{\tvmm{1}\!\!\!^*} \,dx-
                     \int_{0}^{1}  \sigma^1\lre{\tvmm{1}\!\!\!^*} \,dx}
                     }_{2\sigma^2\lre{\tvmm{1}\!\!\!^*}(0) = 2
\dualnorm{\tvmm{1}\!\!\!^*}{\TV_1^2}
= \lambda_2 \frac{\dualnorm{\ud}{TV_1^2}}{\mu_2}
} \\
\equnder{eq:l1l2}& \lambda_2 \TV^2_1\lr{\tvmm{2}} \;.
\end{align*}
In total we obtain
\begin{equation*}
\begin{aligned}
~ &  -\int_{-1}^1 \wtgv \lr{\wtgv-\ud}\,dx \\
=& \frac{{\mu_1}}{\dualnorm{\ud}{\TV_1^1}}
\underbrace{\lambda_1}_{=\frac{\mu_2}{ \dualnorm{\ud}{\TV_1^2}} \mu_1}
TV^1\lr{\tvmm{1}}
+ \lambda_2 TV^2_1\lr{\tvmm{2}}\;.
\end{aligned}
\end{equation*}
Comparing with \eqref{eq:eq1b} and \eqref{eq:eq1a} we have $$\TGV(\wtgv)=-\int
\wtgv\lr{\wtgv-\ud}\,,$$
which together with Lemma \ref{le:char} implies that $\wtgv$ is a minimizer of 
$\G$.
\end{proof}
 
\begin{figure}
\begin{center} 
\includegraphics[]{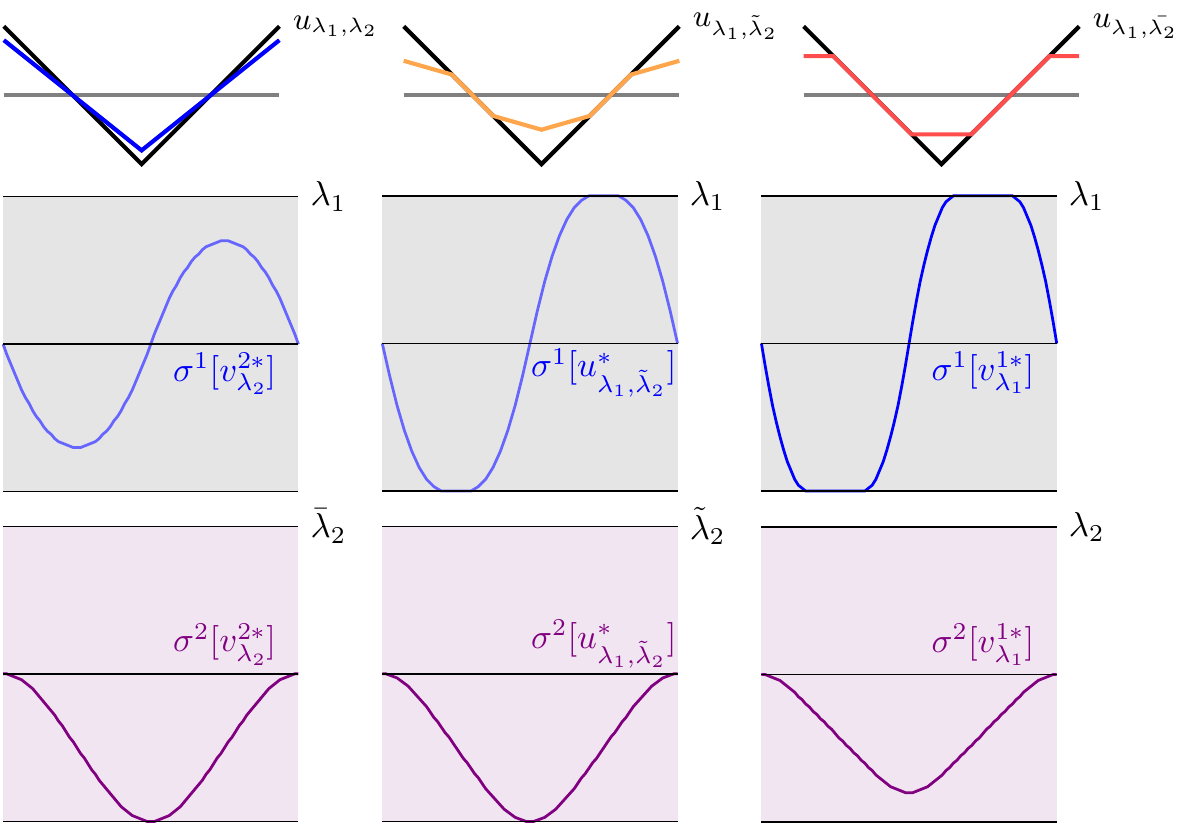} 
\end{center}
\caption{$\utgv$ for fixed $\lambda_1$ and changing $\lambda_2$. 
In this particular case we have $\tilde \lambda_2 \in \Lambda_2$ and 
$\tilde {\lambda_2}, \lambda_2 \not \in \Lambda_2$ with 
$\lambda_2< \tilde \lambda_2 < \bar{\lambda_2}$, 
such that $u_{\lambda_1,\bar{\lambda}_2}=\tvm{1}$ and
$u_{\lambda_1,\lambda_2}=\tvm{2}$.
}
\label{fig:exabs}
\end{figure}

\section{Example 2}
\label{se:ex2}
Consider now as second test-data
\begin{align} \label{eq:udchar}
\ud(x) &= \boldmath{1}_{ \left [-\frac{1}{2},
\frac{1}{2}\right]}(x)-\frac{1}{2}\;, 
\end{align} 
where $\boldmath{1}_{[a,b]}$ is the indicator function of the interval $[a,b]$.
Then 
\begin{equation*}
\dualnorm{\ud}{\TV_1^1}=\frac{1}{4}, \quad 
 \dualnorm{\ud}{\TV_1^2}=\frac{1}{8}\;.
\end{equation*} 
First we calculate minimizers of $\F^i$, as defined in \eqref{eq:l2TV}, in order
to obtain the sets 
$(\lambda_1,\lambda_2)$, where, according to Lemma \ref{le:regions}, the
$\TGV$-minimizers are equal to some $\TV^i$-minimizers.

\subsection{$\F^1$-minimizers}
From \cite{PoeSch08}, we know that
$\tvm{1}:=\lr{1-\frac{\lambda_1}{\dualnorm{\ud}{\TV_1^1}}} \ud$
minimizes $\F^1$ with the test data $\ud$. Furthermore, we have
\begin{equation*}
\dualnorm{\tvmast{1}}{\TV_1^2}=
\lambda_1 \frac{ \dualnorm{\ud}{\TV_1^2}}{\dualnorm{\ud}{\TV_1^1}}.\end{equation*} 
Applying Lemma \ref{le:regions}, we conclude that
$\tvm{1}$ minimizes $\G$ as long as
\begin{align}\label{eq:bound1}
\lambda_2 \geq \dualnorm{v_{\lambda_1}-\ud}{\TV_1^2}=\lambda_1
\frac{ \dualnorm{\ud}{\TV_1^2}}{\dualnorm{\ud}{\TV_1^1}}
=\frac{\lambda_1}{2}\;.
\end{align}

\subsection{$\F^2$-minimizers}
There are 3 different types of solutions (see \cite{PoeSch08} and Figure
\ref{fig:TV2solution}):

\begin{enumerate}
\item For ${\lambda_2} \in [0,\frac{1}{24}(\sqrt{2}\sqrt[4]{3}-\sqrt{3}))$
$\tvm{2}$ is bending four times and $u_{{\lambda_2}}=\ud$ in a region near 
$x=0$. 
$\dualnorm{\tvmast{2}}{\TV_1^1}=\norm{\sigma^1[\tvmast{2}]}_{L^\infty}=g(\lambda_2)$, 
where $g(\lambda_2)$ is a rational of polynomials of higher
order in $\lambda_2$, not written explicitly
here. 
\item 
For ${\lambda_2} \in \lr{\frac{1}{24}(\sqrt{2}\sqrt[4]{3}-\sqrt{3}), 
\frac{1}{24}}$
$\tvm{2}$ is bending at $x_1=\pm
6{\lambda_2} - \frac{1}{4}$, 
and $\tvm{2}=\ud$ in a region near $0$.
 Moreover 
$\dualnorm{\tvmast{2}}{\TV_1^1}=
\frac{1}{18}        
{\frac {1+48\,{{\lambda_2}}+576\,{{{\lambda_2}}}^{2}}{ \left( 1+8\,{
{\lambda_2}} \right) ^{2}}}
$.
\item For 
${\lambda_2} \in \lr{\frac{1}{24}, \frac{1}{8} }$
$\tvm{2}$ is bending once and
$\tvm{2}(x)=\lr{\frac{3}{2}-12{\lambda_2}} \lr{\frac{1}{2}-\abs{x}}$.
Additionally we can calculate 
$\dualnorm{\tvmast{2}}{\TV_1^1}=\frac{1}{16}+\frac{3}{2}{\lambda_2}$.
\end{enumerate}

The expressions $\dualnorm{\tvm{2} -\ud}{\TV_1^1}$ are used to calculate the set 
$\Lambda$, the set of parameters, where the $\G$-minimizer might be different 
to the $\F^1$ or $\F^2$-minimizer.

\begin{figure}
\begin{center} \includegraphics[]{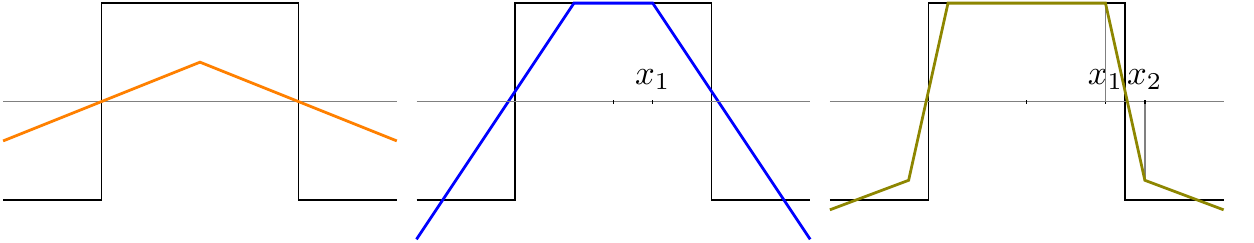} \end{center}
\caption{Three different types of $\F^2$ minimizers $\tvm{2}$: 
They can bend once (3), twice (2) or four times (1).}
\label{fig:TV2solution}
\end{figure}
We write the solutions in the form
\begin{align}\label{eq:tvm1}
 \tvm{2}=
\begin{cases}
 \ud(x) &\text{for} \abs{x}< x_1\\
 k_1 \abs{x} + d_1 &\text{for } x_1 \leq  \abs{x} < x_2\\
 k_2 \abs{x} + d_2 &\text{for } x_2 \leq \abs{x} \leq 1\
\end{cases}
\end{align}
keeping in mind that $x_1$ can be 0 (third case), or $x_2$ can be larger 
then one (third and second case).
$\tvm{1}$ is bending at $x_i$, such that $\sigma^2[\tvmast{1}]$ is extremal  at $x_0$ (hence $\pm \lambda_2$).
Thus, the coefficients $d_i,k_i$ are determined by the following equations:
\begin{align*}
 \sigma^2\lre{\tvmast{1}}(x_1)&=\lambda_2, \; \sigma^1\lre{\tvmast{1}}(x_1) = 0
&&\text{cases 1,2,3}\\
 \sigma^2\lre{\tvmast{1}}(x_2)&=-\lambda_2,\;  \sigma^1\lre{\tvmast{1}}(x_2)=0
&&\text{case 1}\;.
\end{align*}

\subsection{$\G$-minimizers}
We consider the same approach as for the previous example. 
Hence, first we calculate the set $\Lambda$ as in Definition
\ref{de:Lambda}, which is illustrated as the green (solid) set in 
Figure \ref{fig:region2}. 
We have:
\begin{equation*}
\begin{aligned}
\partial \Lambda & = \set{
  \lr{\lambda_1, \dualnorm{\tvmast{1}}{\TV_1^2}} : 
\lambda_1 \in \lre{0,\dualnorm{\ud}{\TV_1^1}}} \\
& \qquad \cup
\set{
\lr{\dualnorm{\tvmast{2}}{\TV_1^1},\lambda_2}: 
\lambda_2 \in \lre{ 0, \dualnorm{\ud}{\TV_1^2}}}\,.
\end{aligned}
\end{equation*}
Next we set up a general 
Ansatzfunction $\wtgv$ of piecewise affine functions, 
that is bending, once, twice or four times and jumping at 
$x=\pm 0.5$. Setting $\wtgv^*=\wtgv-\ud$, 
we find the coefficients (of the piecewise affine functions) by solving a 
number of non-linear equations coming
from the conditions - $\abs{\sigma^1[\wtgv^*](x)}=\lambda_1$ whenever the
Ansatzfunction
jumps and $\abs{\sigma^2[\wtgv^*](x)}= \lambda_2,\sigma^1[\wtgv^*](x)=0$,
whenever the Ansatzfunction bends. 
We omit the explicit formulas and further calculations. 

Then for $(\lambda_1,\lambda_2) \in \Lambda$ given, we found that the minimizers of
$\G$ can be written in a compact form:

\begin{theorem}
Let $(\lambda_1,\lambda_2) \in \Lambda$  and
${\mu_1},{\mu_2}$ be such that 
\begin{align*}
\lambda_2&=\frac{{\mu_1} {\mu_2}}{\dualnorm{\ud}{\TV_1^1}}, 
&\lambda_1&=
\frac{{\mu_1}}{\dualnorm{\ud}{\TV_1^1}} \dualnorm{\tvmm{2}\!\!\!^*}{\TV_1^1} \;.
\end{align*}
Then for $\ud$ as in \eqref{eq:udchar}
\begin{equation}\label{eq:sum2}
\utgv = \tvmm{1} +\frac{{\mu_1}}{\dualnorm{\ud}{\TV_1^1}}
\tvmm{2}
\end{equation}
minimize $\G$.
\end{theorem}

\begin{proof} Lemma \ref{le:min} states that $\utgv$ is a
minimizer if \newline
$\TGV(\utgv) = - \int_{-1}^1 \utgv (\utgv-\ud)\,dx$ (see 
\eqref{eq:optiTGV}) and $\utgv^* \in \dualball{\TGV}$.
Using \eqref{eq:sum2} we can estimate $\TGV(\utgv)$ by 
\begin{equation*}
\begin{aligned}
\TGV(\utgv) 
&\leq \TGV\lr{\tvmm{1}} + \frac{{\mu_1}}{\dualnorm{\ud}{\TV_1^1}} \TGV\lr{\tvmm{2}}\\
&\leq 
 \TV^1_{\lambda_1}\lr{\tvmm{1}}+
  \frac{{\mu_1}}{\dualnorm{\ud}{\TV_1^1}} \TV^2_{\lambda_2}
  \lr{\tvmm{2}} \;.
\end{aligned}
\end{equation*}
Since $\tvmm{1}=\lr{1-\frac{\mu_1}{\dualnorm{\ud}{\TV_1^1}}}\ud$, we have
\begin{equation*}
 \utgv^* =\frac{ \mu_1 }{ \dualnorm{\ud}{\TV_1^1} } \tvmm{2}\!\!\!^*\;.
\end{equation*}
Note that from the choice of the parameters $\mu_1,\mu_2$ we have
\begin{align*}
\lambda_1 &=
\lambda_1 \dualnorm{\utgv^*}{\TV_{\lambda_1}^1}\\
&=
\norm{\sigma^1\lre{\tvmm{1}+\frac{\mu_1}{\dualnorm{\ud}{\TV_1^1} } \tvmm{2}-\ud }}_{L^\infty }\\ 
&= 
\frac{\mu_1}{\dualnorm{\ud}{\TV_1^1} }
\norm{\sigma^1\lre{\tvmm{2}-\ud }}_{L^\infty }\\
&=
\frac{\mu_1}{\dualnorm{\ud}{\TV_1^1} }\lambda_1 \dualnorm{\tvmm{2}\!\!\!^*}{\TV_{\lambda_1}^1}
\end{align*}
and also
\begin{equation*}
\begin{aligned}
\lambda_2 &=
\lambda_2 \dualnorm{\utgv^*}{\TV_{\lambda_2}^2}\\
&= 
\frac{\mu_1}{\dualnorm{\ud}{\TV_1^1} }
\norm{\sigma^2\lre{\tvmm{2}-\ud }}_{L^\infty }\\
&=
\frac{\mu_1}{\dualnorm{\ud}{\TV_1^1} }\underbrace{\lambda_2 \dualnorm{\tvmm{2}\!\!\!^*}{\TV_{\lambda_2}^2}}_{\mu_2}\;,
\end{aligned}
\end{equation*}
such that $\utgv^*\in \dualball{\TGV}$ and 
\begin{equation}\label{eq:eq1}
-\int_{-1}^1 \utgv (\utgv-\ud)\,dx \leq \TGV(\utgv) \;.
\end{equation}
Moreover, we can write
\begin{align*}
~&  \int_{-1}^1 \utgv \lr{\utgv - \ud} \,dx\\
=& 
\int_{-1}^1 \tvmm{1} \lr{\utgv - \ud} \,dx+ 
\frac{{\mu_1}}{\dualnorm{\ud}{\TV_1^1}} \int_{-1}^1 \tvmm{2} \lr{\utgv - \ud}\,dx.
\end{align*}
Since $\tvmm{1} = \lr{1- \frac{{\mu_1}}{\dualnorm{\ud}{\TV_1^1}} }\ud$, 
and $\ud(x) \in \set{-\frac{1}{2},\frac{1}{2}}$ 
we obtain
\begin{align*}
~ & 
\int_{-1}^1 \tvmm{1} \lr{\utgv - \ud} \,dx\\
= & \frac{1}{2} \lr{1- \frac{{\mu_1}}{\dualnorm{\ud}{\TV_1^1}}}\\
&\quad \lr{ 
- \int_{-1}^{-\frac{1}{2}} \lr{\utgv - \ud}   \,dx
+ \int_{-\frac{1}{2}}^{\frac{1}{2}} \lr{\utgv - \ud} \,dx
- \int_{\frac{1}{2}}^{1} \lr{\utgv - \ud} \,dx } 
\\
= &\frac{1}{2} \lr{1- \frac{{\mu_1}}{\dualnorm{\ud}{\TV_1^1}} }
\left( - \sigma^1[\utgv^*]\lr{-\frac{1}{2}}
+\sigma^1[\utgv^*]\lr{\frac{1}{2}}
-\sigma^1[\utgv^*]\lr{-\frac{1}{2}} \right.\\
&\phantom{aaaaaaaaaaaaaaaa}
\left.
-\sigma^1[\utgv^*]\lr{1}
+\sigma^1[\utgv^*]\lr{\frac{1}{2}}
\right)\;.
\end{align*}
Now by the choice of the parameter $\lambda_1$ we have
$\sigma^1[\utgv^*]\lr{-\frac{1}{2}} = +\lambda_1$ and 
\newline
$\sigma^1[\utgv^*]\lr{\frac{1}{2}} = -\lambda_1$ 
such that the equation above simplifies to
\begin{equation*}
-\int_{-1}^1 \tvmm{1} \lr{\utgv - \ud}\,dx 
=\lambda_1 \, 2 \lr{1- \frac{{\mu_1}}{\dualnorm{\ud}{\TV_1^1}} }
=\lambda_1 TV^1\lr{\tvmm{1}}
\;.
\end{equation*}
Next it remains to calculate
$\int_{-1}^1 \tvmm{2} \lr{\utgv-\ud}\,dx $.
Since $$\utgv-\ud = \tvmm{1}- \frac{\mu_1}{\dualnorm{\ud}{\TV_1^1}} \tvmm{2}
-\ud = \frac{\mu_1}{\dualnorm{\ud}{\TV_1^1}} \lr{\tvmm{2}-\ud}\,,$$ we have
\begin{align*}
 \int_{-1}^1 \tvmm{2} \lr{\utgv-\ud}\,dx &=
 \frac{\mu_1}{\dualnorm{\ud}{\TV_1^1}} \int_{-1}^1 \tvmm{2} \lr{\tvmm{2}-\ud}\,dx \\
 &=
 - \frac{\mu_1}{\dualnorm{\ud}{\TV_1^1}} \mu_2 TV^2_1(\tvmm{2}) \;
\end{align*}
where we used $TV^2_1(\tvmm{2}) = - \mu_2 \int_{-1}^1 \tvmm{2}\lr{\tvmm{2} - 
\ud}  
\;dx$, 
the optimality condition for $\F^2$-minimizers as in \eqref{eq:optiTV}.
Hence in total, using the connections between $\lambda_i$ and $\mu_i$, we obtain
\begin{align*}
- \int_{-1}^1 \utgv \lr{\utgv-\ud}  \,dx
&=
   \lambda_1 \TV^1_{1}(\tvmm{1}) 
+ \lambda_2 \frac{{\mu_1}}{\dualnorm{\ud}{\TV_1^1}} 
\TV^2_{1}(\tvmm{2}) \;.
\end{align*}
A Comparison with \eqref{eq:eq1} shows that
$\TGV \lr{\utgv} = - \int_{-1}^1 \utgv(\utgv-\ud)\,dx$, 
hence according to Lemma \ref{le:min}
$\utgv$ minimizes $\G$.
\end{proof}

\begin{figure}
\begin{center} 
\includegraphics[scale=0.7]{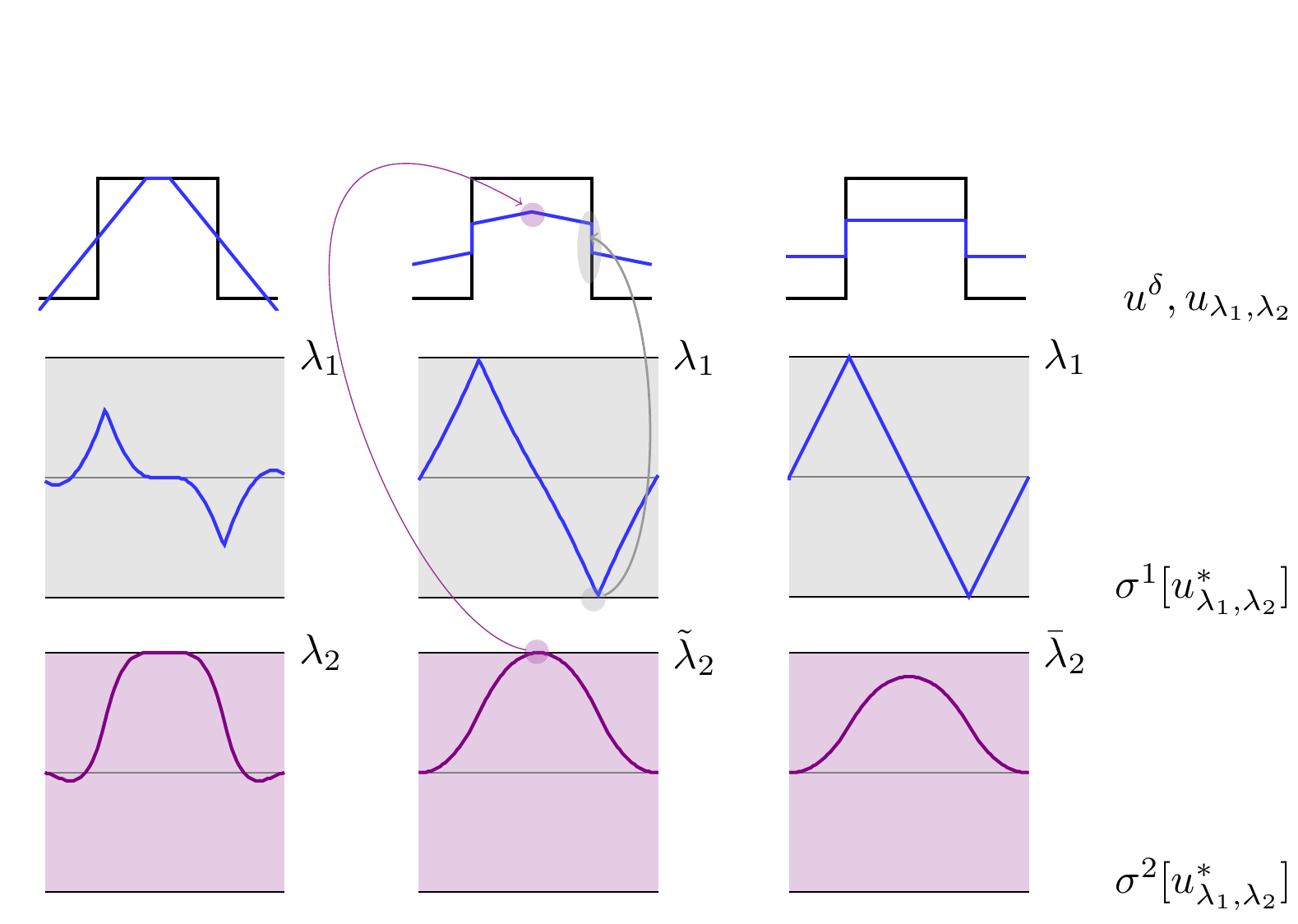} 
\end{center}
\caption{$\utgv$ for fixed $\lambda_1$ and changing $\lambda_2$. 
In this particular case we have $(\lambda_1,\tilde \lambda_2) \in \Lambda$ and 
$(\lambda_1,\tilde {\lambda_2}), (\lambda_1,\lambda_2) \not \in \Lambda$ with 
$\lambda_2< \tilde \lambda_2 < \bar{\lambda_2}$, 
such that $u_{\lambda_1,\bar{\lambda}_2}=\tvm{1}$ and
$u_{\lambda_1,\lambda_2}=\tvm{2}$.
}
\label{eq:exchar}
\end{figure}

\begin{figure}
\begin{center} \includegraphics[scale=0.9]{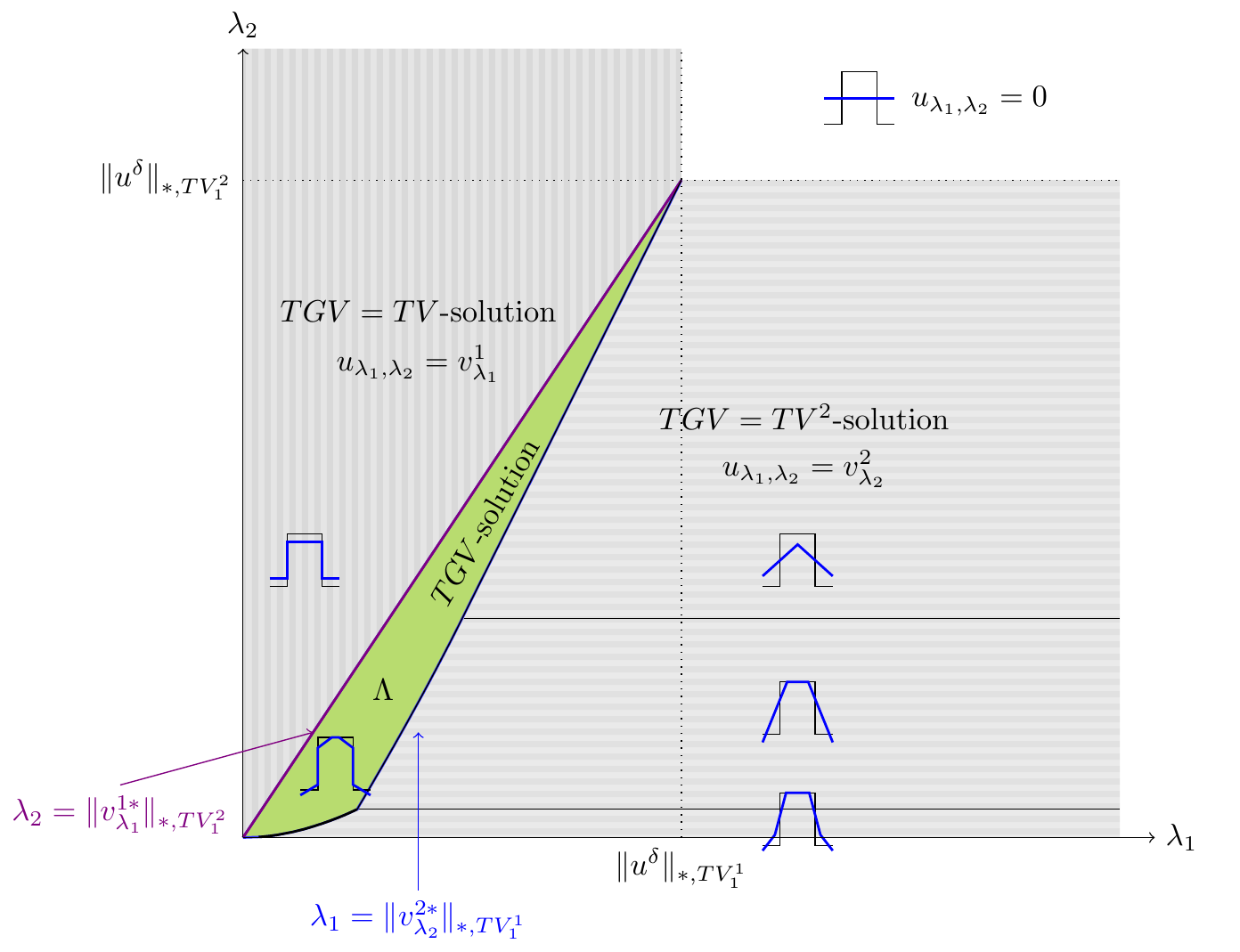} 
\end{center}
\caption{
Solid green region: $\Lambda$, here $\utgv\not= \tvm{i}$.
Gray region: $\utgv = \tvm{i}$, vertical lines: $i=1$,
horizontal lines: $i=2$. 
}
\label{fig:region2}
\end{figure}

\section{Example III}\label{se:ex3}
Finally we consider $\ud= x^2 - \frac{1}{3}$ but only sketch the different
minimizers of $\F_i$ and $\mathcal{G}$ in order to show that in general,
minimizers of $\mathcal{G}$ cannot be written as a sum of $\mathcal{F}_i$-
minimizers.
We have $ \dualnorm{\ud}{\TV_1^2}=\frac{1}{12}=0.833,
\dualnorm{\ud}{\TV_1^1}= \frac{2\sqrt{3}}{27}  \sim 0.12$.

\subsection{$\F^1$-minimizers.} 
Since $\ud$ is continuous, also the $\F_{1}$-minimizer is continuous. 
From the characterization of $\F_{1}$ minimizers we know that 
$\tvm{1}$ is either equal to $\ud$ in an interval $(\pm c_1, \pm c_2)$ or 
constant $\ud(c_1),\ud(c_2)$ in the other intervals. 
In Figure \ref{fig:example3} (left), we ilustrate $\tvm{1}$ for different values 
of $\lambda_1$.

\subsection{$\F^2$-minimizers.} 
In this case, we have to consider two different types of minimizers. 

\begin{itemize}
 \item  $\lambda_2$ large: (that is $\lambda \sim  \dualnorm{\ud}{\TV_1^2}$), 
$\tvm{2}$ is piecewise constant and bending at $x=0$.
Such solutions are constructed by considering the 
Ansatzfunctions $w(x,k):=k\lr{\abs{x}-\frac{1}{2}}$. The parameter $k$ is 
determined such that $\abs{ \sigma^2\lre{w(\cdot,k)-\ud}(0)}=\lambda_2$ (the 
Ansatzfunction $w$ is bending at $x=0$, hence the $\sigma^2$ of the dual 
minimizer has be extremal, hence equal to $\lambda_2$).
This Ansatzfunction works until for some $\lambda_2=\bar \lambda_2$, we have
$\ud(0)=w(0,k)$. 

\item Then for $\lambda_2\leq \bar \lambda_2$, we use a different 
Ansatzfunction $w$ that satisfies: $w(x)=\ud(x)$ for $x\in (-c,c)$ and some 
$c>0$ and 
$w(x)$ is affine linear in $(-1,-c) \cup (c,1)$ and continuous at $x=\pm c$.
The coefficients are determined such that 
$w \in \mathcal{H}^2$ and 
$\abs{ \sigma^2\lre{w-\ud}(c)}=\lambda_2$. 
\end{itemize}
We illustrated both types of solutions in Figure \ref{fig:example3} (right).

\subsection{$\mathcal{G}$-minimizers.}
For $(\lambda_1,\lambda_2) \in \Lambda$ as in Definition \ref{de:Lambda}, 
we set up an Ansatzfunction that satisfies the following:
\begin{itemize}
 \item $w$ is continuous,
 \item $w(x)=w(-x)$,
 \item $w(x)=\ud(x)$ for $x \in (c_2,c_3)$ and $0\leq c_2< c_3\leq 1$, 
 \item either $w$ is bending at $x=0$, or $w(x)=\ud(x)$ for $x \in [0,c_1)$ 
with $c_1<c_2$, 
 \item $w$ is piecewise affine linear else.
\end{itemize}

We illustrate minimizers for different choices of $(\lambda_1,\lambda_2)$  
in Figure \ref{fig:example3}.

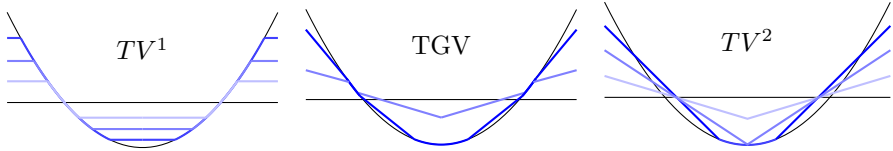
\begin{figure}
\begin{tikzpicture}[domain=-1:1,scale=1.8]

\draw (-1,0) -- (1,0);
\draw[color=black, samples=50,yshift=-0.333 cm] plot (\x,{pow(\x,2)})
node[right] {}; 

\foreach \x/\y/\a in {0.24/0.9/75,0.37/0.8/50,0.47/0.7/25}
\foreach \i in {1,-1}{

\draw[domain=\x:\y,color=blue!\a!, thick, samples=50,yshift=-0.333
cm,xscale=\i] plot
(\x,{pow(\x,2)}) node[right] {}; 
\draw[yshift=-0.333 cm,xscale=\i,color=blue!\a!, thick] (\x,\x*\x) --
(0,\x*\x);
\draw[yshift=-0.333 cm,xscale=\i,color=blue!\a!, thick] (\y,\y*\y) --
(1,\y*\y);
}
\node at (0,0.4) {$TV^1$};
\end{tikzpicture}
\begin{tikzpicture}[domain=-1:1,scale=1.8]
\draw (-1,0) -- (1,0);
\draw[color=black, samples=50,yshift=-0.333 cm] plot (\x,{pow(\x,2)})
node[right] { };

\draw[domain=-0.2:0.2,color=blue, thick, samples=50,yshift=-0.333 cm] plot
(\x,{pow(\x,2)}) node[right] {}; 
\foreach \i in {1,-1}
{
\draw[yshift=-0.333 cm,xscale=\i,color=blue, thick] (0.2,0.04) -- (0.6,0.36);
\draw[domain=0.6:0.7,color=blue, thick, samples=50,yshift=-0.333 cm,xscale=\i]
plot (\x,{pow(\x,2)}) node[right] {}; 
\draw[yshift=-0.333 cm,xscale=\i,color=blue, thick] (0.7,0.49) -- (1,0.85);
}

\foreach \i in {1,-1}
{
\draw[yshift=-0.333 cm,xscale=\i,color=blue!50!, thick] (0.0,0.2) --
(0.62,0.62*0.62);
\draw[domain=0.6:0.7,color=blue, thick, samples=50,yshift=-0.333 cm,xscale=\i]
plot (\x,{pow(\x,2)}) node[right] {}; 
\draw[yshift=-0.333 cm,xscale=\i,color=blue!50!, thick] (0.68,0.68*0.68) --
(1,0.55);
}

\node at (0,0.4) {TGV};
\end{tikzpicture}
\begin{tikzpicture}[domain=-1:1,scale=1.9]
\node at (0,0.4) {$TV^2$};
\def\c{0.2}
\def\k{0.992}
\def\a{0.8336}
\draw (-1,0) -- (1,0);
\draw[color=black, samples=50,yshift=-0.333 cm] plot (\x,{pow(\x,2)})
node[right] { }; 
\draw[domain=-0.2:0.2,color=blue!75!, thick, samples=50,yshift=-0.333 cm] plot
(\x,{pow(\x,2)}) node[right] {}; 
\foreach \i in {1,-1}
\draw[yshift=-0.333 cm,xscale=\i,color=blue, thick] (\c,\c*\c) --
(1,\a);
\draw[yscale=0.3,blue!25!,thick] (-1,0.5) -- (0,-0.5) -- (1,0.5);
\draw[yscale=0.66,blue!50!,thick] (-1,0.5) -- (0,-0.5) -- (1,0.5);
\end{tikzpicture}
\caption{minimizers of $\F_1$ (left), $\G$ (middle), $\F_2$ (right), 
for different parameters.
}\label{fig:example3}
\end{figure}

\section*{Acknowledgments}

This work has been supported by the Austrian Science Fund (FWF) within the national research networks Photoacoustic Imaging in Biology and Medicine (project S10505) and 
Geometry \& Simulation (project S11704).
C. P. acknowledges support by the Austrian Science Fund
(FWF) within the Schr\"odinger program (project J-297).

\def\cprime{$'$} \providecommand{\noopsort}[1]{}\def\cprime{$'$}

\end{document}